\documentclass[twoside,11pt]{article}

	\usepackage[T1]{fontenc}	% Accents codés dans la fonte.
	\usepackage[latin1]{inputenc}	% Accents 8 bits dans le fichier.
	\usepackage{vmargin}		% Régler la taille de la feuille.
	\usepackage{fancyhdr}		% Régler le titre courant et le bas de page.
	\usepackage{amsfonts,amsthm,amsmath,stmaryrd, amssymb,mathrsfs} %pour les maths
	\usepackage{graphicx, subfigure} %pour les graphiques
	\usepackage{mfpic} %pour le graphe de fonctions et autres graphiques par METAFONT
	\usepackage{listings} %pour écrire le code programmation
	\setpapersize{custom}{21cm}{29.7cm}
	\setmarginsrb{25mm}{10mm}{25mm}{47mm}{10mm}{6mm}{0mm}{10mm} %marge 1 (gauche) et 3 (droite) à 30.5
	\usepackage[pdftex]{hyperref}
	\usepackage{color}
	
	\numberwithin{equation}{section}
	
\def\e{\varepsilon}

\def\R{{\mathbb R}}
\def\C{{\mathbb C}}
\def\N{{\mathbb N}}
\def\Z{{\mathbb Z}}
\def\d{\partial}
\def\a{\alpha}
\def\EE{{\mathbb E}}
\def\g{\gamma}

\def\u{{\mathbf u}}
\def\v{{\mathbf v}}
\def\w{{\mathbf w}}
\def\f{{\mathbf f}}
\def\h{{\mathbf h}}
\def\H{{\mathbf H}}
\def\A{{\mathbf A}}
\def\F{{\mathbf F}}
\def\RR{{\mathbf R}}
\def\dsp{\displaystyle}
\def\be{\begin{equation}}
\def\ee{\end{equation}}
\def\Id{{\rm Id}}

\def\und{\underline}

\begin{document}

\theoremstyle{plain}
\newtheorem{theo}{Theorem}[section]

\theoremstyle{plain}
\newtheorem*{theo*}{Theorem}

\theoremstyle{plain} %mettre definition sinon
\newtheorem{defi}[theo]{Definition}

\theoremstyle{plain}
\newtheorem{prop}[theo]{Proposition}

\theoremstyle{plain}
\newtheorem{coro}[theo]{Corollary}

\theoremstyle{plain}
\newtheorem{lemma}[theo]{Lemma}

\theoremstyle{definition}
\newtheorem*{condi}{Condition}

\theoremstyle{definition}
\newtheorem{notation}[theo]{Notation}

\theoremstyle{plain} %mettre definition sinon
\newtheorem{remark}[theo]{Remark}

\theoremstyle{plain} %mettre definition sinon
\newtheorem{hypo}[theo]{Assumption}

\title{On hyperbolicity and Gevrey well-posedness. \\ Part two: Scalar or degenerate transitions}

\author{Baptiste Morisse \thanks{School of Mathematics, Cardiff University - \href{mailto:morisseb@cardiff.ac.uk}{morisseb@cardiff.ac.uk}. The author is supported by the EPSRC grant "Quantitative Estimates in Spectral Theory and Their Complexity" (EP/N020154/1). 
The author thanks his PhD advisor Benjamin Texier for all the remarks on this work, Jeffrey Rauch for interesting discussions and Jean-François Coulombel for his careful reading of a previous version.}
}
\date{\today}

\maketitle

\begin{abstract} 

For first-order quasi-linear systems of partial differential equations, we formulate an assumption of a transition from initial hyperbolicity to ellipticity. This assumption bears on the principal symbol of the first-order operator. Under such an assumption, we prove a strong Hadamard instability for the associated Cauchy problem, namely an instantaneous defect of Hölder continuity of the flow from $G^{\sigma}$ to $L^2$, with $0 < \sigma < \sigma_0$, the limiting Gevrey index $\sigma_0$ depending on the nature of the transition. We restrict here to scalar transitions, and non-scalar transitions in which the boundary of the hyperbolic zone satisfies a flatness condition. As in our previous work for initially elliptic Cauchy problems [B. Morisse, \textit{On hyperbolicity and Gevrey well-posedness. Part one: the elliptic case}, arXiv:1611.07225], the instability follows from a long-time Cauchy-Kovalevskaya construction for highly oscillating solutions. This extends recent work of N. Lerner, T. Nguyen, and B. Texier [\textit{The onset of instability in first-order systems}, to appear in J. Eur. Math. Soc.].

\end{abstract}

%= e^{-\e^{-\delta}} \und{u}(t,x,x\cdot\xi_0/\e)$

\newpage

%%%%%%%%%%%%%%%%%%%%%%%%%%%%%%%%%%%%%%%%%%%%%%%%%%%%%%%%%%%%%%%%%%%%%%%%%%%%%%%
%%%%%%%%%%%%%%%%%%%%%%%%%%%%%%%%%%%%%%%%%%%%%%%%%%%%%%%%%%%%%%%%%%%%%%%%%%%%%%%

%%%%%%%%%%%%%%%%%%%%%%%%
\setcounter{tocdepth}{2}
\tableofcontents
%%%%%%%%%%%%%%%%%%%%%%%%

%\newpage

%%%%%%%%%%%%%%%%%%%%%%%%%%%%%%%%%%%%%%%%%%%%%%%%%%%%%%%%%%%%%%%%%%%%%%%%%%%%%%%

%%%%%%%%%%%%%%%%%%%%%%%%%%%%%%%%%%%%%%%%%%%%%%%%%%%%%%%%%%%%%%%%%%%%%%%%%%%%%%%
%%%%%%%%%%%%%%%%%%%%%%%%%%%%%%%%%%%%%%%%%%%%%%%%%%%%%%%%%%%%%%%%%%%%%%%%%%%%%%%

\section{Introduction}

We consider the following Cauchy problem, for first-order quasi-linear systems of partial differential equations:
\be
	\label{Cauchy}
	\d_{t}u = \sum_{j=1}^{d} A_{j}(t,x,u) \d_{x_j}u + f(t,x,u) \;,\qquad u(0,x) = h(x) .
\ee

\noindent The system is of size $N$, that is $u(t,x)$ and $f(t,x,u)$ are in $\R^N$ and the $A_j(t,x,u)\in\R^{N\times N}$. The time $t$ is nonnegative, and $x$ is in $\R^{d}$. We assume throughout the paper that the $A_j$ and $f$ are analytic in a neighborhood of some point $(0,x_0,u_0)\in\R_{t}\times\R^{d}_{x}\times\R^{N}_{u}$. 

Under assumptions of weak defects of hyperbolicity for the first-order operator, we prove ill-posedness of \eqref{Cauchy} in Gevrey spaces. Weak defect of hyperbolicity is here understood as a transition from hyperbolicity of the principal symbol at initial time, to ellipticity of the principal symbol for later times. Our results extend Métivier's ill-posedness theorem in Sobolev spaces for initially elliptic operators \cite{metivier2005remarks}, our own ill-posedness result in Gevrey spaces for initially elliptic operators \cite{morisse2016I}, Lerner, Nguyen and Texier's theorem on systems transitioning from hyperbolicity to ellipticity \cite{lerner2015onset}, and echo Lu's construction of WKB profiles \cite{lu2016resonances} which are destabilized by terms not present in the initial data. 

Our proofs use Métivier's method developed in \cite{metivier2005remarks} based on majoring series, hence the assumption of analyticity for the $A_j$ and $f$. Our assumptions of weak defects of hyperbolicity mean that the operator in \eqref{Cauchy} experiences a transition in time from hyperbolicity to non hyperbolicity. The transition is possibly not uniform in space.  Our assumptions bearing on the principal symbol, and the associated normal forms, are presented in Section \ref{section.assumptions}. Our results are Theorems \ref{theorem.smooth} and \ref{theorem.airy}, stated in Section \ref{subsection.results}. The proofs comprise Sections \ref{section.ansatz} to \ref{section.existence}.

In the companion paper \cite{morisse2016IIz}, we consider the case of genuinely non-scalar transitions.

%The different assumptions on the principal symbol and the normal forms associated will be developed in Section \ref{section.assumptions}, followed by the associated results \ref{theorem.smooth} and \ref{theorem.airy}.

%%%%%%%%%%%%%%%%%%%%%%%
\subsection{Background}
\subsubsection{A long-time Cauchy-Kovaleveskaya result for elliptic Cauchy problems}
\label{subsubsection.metivier}
%%%%%%%%%%%%%%%%%%%%%%%%%%%%%%%%%%%%%%%%%%%%%%%%%%%%%%%%%%%%%%%%%%%%%%%%%%%%%%%%%%%%%

%The elliptic case.
%Tu dis quen gros dans le cas elliptique (à d´ efinir) cest mal pos´ e: dire comment,
%expliquer que Metivier est plus fin (non existence de solutions, etc) mais quil reste
%en Sobolev, que dans le papier 1 tu as trait´ e de cas Hadamard en Gevrey, dire un
%mot sur le fait que le temps dobservation est long et que cest la difficult´ e principale
%du papier 1.

Our article \cite{morisse2016I} contains a long-time Cauchy-Kovalevskaya theorem, based on the paper \cite{metivier2005remarks} by Métivier, which proves an Hadamard instability result for initially elliptic quasi-linear systems in Gevrey spaces. Precisely, the result of \cite{morisse2016I} asserts that the flow associated to the Cauchy problem \eqref{Cauchy} fails to be Hölder from a highly regular $\sigma$-Gevrey space to the very lowly regular $L^2$ space, locally in the $x$ variable and for $\sigma$ less than a critical exponent $\sigma_0$ depending on initial spectrum, under the assumption of initial ellipticity for the first-order differential operator. 

Here initial ellipticity is understood as an initial defect of hyperbolicity. That is for some $(x_0,\vec{u}_0,\xi_0) \in \R_{x}^{d}\times\R_{u}^{N}\times\R_{\xi}^{d},$ the principal symbol at $(0,x_0,\vec{u}_0,\xi_0)$:
$$
	A_0 := \sum_{j=1}^{d} A_j(0,x_0,\vec{u}_0) \xi_{0,j}
$$

\noindent has at least one couple of non-real eigenvalues, with imaginary part $\pm i \g_0$ associated to eigenvectors $\vec{e}_{\pm}$. 

%This implies that the linear approximation 
%$$
	%\d_{t} u - \sum_j A_j(0,x_0,\vec{u}_0) \d_{x_j}u = 0
%$$
%
%\noindent satisfies, for initial conditions of the type $h_{\e}(x) = e^{i (x-x_0)\cdot\xi_0/\e } \vec{e}_{+} + e^{ - i(x-x_0)\cdot \xi_0/\e} \vec{e}_{-}$, the exponential growth 
%$$
	%u_{\e}(t,x) := e^{t \sum_j A_j(0,x_0,\vec{u}_0) \d_{x_j} } h_{\e}(x) \approx e^{\g_0 t/\e} .
%$$
%
%\noindent As the $\sigma$-Gevrey norm of the analytic initial condition $h_{\e}$ is of order $e^{c \e^{-\sigma}}$ for some $c>0$ (see Lemma 3.3 of \cite{morisse2016I}), we note that the Hölder ratio for $h_{\e}$ and its corresponding solution $u_{\e}$ satisfies the lower bound
%\be
	%\label{intro.hadamard}
	%\frac{\dsp ||u_{\e}||_{L^2(B_{r_0}\times[\und{t} - \e , \und{t}])} }{ \dsp ||h_{\e}||_{G^{\sigma}( B_{r_1} ) }^{\a} } \gtrsim \frac{ \dsp e^{\g_0 \und{t} /\e} }{ \dsp e^{\a c \e^{-\sigma}} }
%\ee
%
%\noindent for some final time $\und{t}$ larger than $\e$. If $\und{t}/\e$ is larger than $\e^{-\sigma}$, then the right-hand side tends to $+\infty$ as $\e\to0$. This means that, for sufficiently large time (for the rescaled time $s = t/\e$), we get the Hadamard instability for the linear approximation on highly oscillating and well-polarized initial data $h_{\e}$.

%To extend this formal result on a very simple approximation of \eqref{Cauchy}, 

In \cite{morisse2016I} we posit in $(3.2)$ the ansatz $u_{\e}(t,x) = \e\u(t /\e , x , (x-x_0)\cdot \xi_0/\e)$, where $\u(s,x,\theta)$ is periodic in the $\theta$ variable. We transform then the Cauchy problem \eqref{Cauchy} into the equation
\be
	\label{intro.equation.ds}
	\d_{s} \u - \und{A}(\e s , x) \d_{\theta} \u  = G(\u)
\ee

\noindent for some non-linear remainder term $G(\u)$. The leading term $\und{A}(t,x)$ is here the principal symbol
\be
	\label{intro.def.und.A}
	\und{A}(t,x) = \sum_j A_j(t,x,\vec{u}_0) \xi_{0,j} .
\ee

\noindent The ellipticity condition is an open condition bearing on the principal symbol $\und{A}$. In particular, ellipticity at $(0,x_0)$ implies ellipticity around $(0,x_0)$. The proof of \cite{morisse2016I} introduces the propagator $U$ defined by 
\be
	\label{intro.equation.propa}
	\d_{s} U(s',s,x,\theta) - \und{A}(\e s, x) \d_{\theta} U(s',s,x,\theta) = 0 \quad , \quad U(s',s',x,\theta) = {\rm Id}.
\ee
	
\noindent By ellipticity, the propagator $U$ has an exponential growth. We introduce an appropriate Banach space of functions of $(s,x,\theta)$ which are analytical in the $x$ variable and whose Fourier coefficients in $\theta$ have an exponential growth which reflects the growth of the propagator. A fixed point argument shows existence and uniqueness, and exponential growth in this space, which implies the Hadamard instability.

The main issue in \cite{morisse2016I}, compared to the previous analysis of Metivier \cite{metivier2005remarks}, is that in Gevrey spaces, the Hadamard instability is recorded at much longer times than in Sobolev spaces. The instability is observed thanks to highly oscillating, well-polarized initial data, which generate solutions growing exponentially both in time and frequency. Observing an instability means that at some time, the $L^2$ norm of the solution is far greater (with respect to the frequency) than the Sobolev or Gevrey norm of the initial datum. Considering the fundamental oscillation $e^{i x \cdot \xi}$ with frequency $\xi$, a simple computation leads to the Sobolev norm $||e^{i x \cdot \xi}||_{H^m} \simeq |\xi|^{m}$, whether the Gevrey norm is $|| e^{ix\cdot\xi} ||_{\sigma,c,K} \simeq e^{ (\sigma c^{\sigma})^{-1}  |\xi|^{\sigma}} $ (see \eqref{def.norm.gevrey} below and Lemma 3.3 in \cite{morisse2016I}). Hence the observation of the instability is recorded as a much longer time in Gevrey spaces than in Sobolev spaces.

\subsubsection{Lerner, Morimoto and Xu's result on transition to ellipticity for scalar equations}
%%%%%%%%%%%%%%%%%%%%%%%%%%%%%%%%%%%%%%%%%%%%%%%%%%%%%%%%%%%%%%%%%%%%%%%%%%%%%%%%%%%%%%%%%%%%%%%%%%%

%Following the above result of Métivier on initial ellipticity, several articles obtained some results on transitions in time to ellipticity, assuming initial hyperbolicity. 

In \cite{lerner2010instability}, Lerner, Morimoto and Xu introduce the notion of transition to ellipticity for initially hyperbolic systems. A prototypical example is the Burgers equation with a complex forcing: $ \d_t u + u\d_{x}u = i $. In the case of real data, the principal symbol is initially hyperbolic. Due to the complex forcing, the principal symbol is elliptic for ulterior times. For general equations \eqref{Cauchy} (with $N = 1$: scalar equations), under bracket conditions generalizing the situation for Burgers with complex forcing, and describing a transition from hyperbolicity to ellipticity, the authors in \cite{lerner2010instability} prove a strong form of instability, namely that if local $C^2$ solutions exist, then the complement of the analytic wave-front set of the datum is not empty. In particular, if the bracket conditions are formulated at $(x_0,u_0)\in\R^{d}\times\R$, it is shown in \cite{lerner2010instability} that for any analytical datum $\und{h}$ such that $\und{h}(x_0) = u_0$ , there exists smooth initial data $h$ close to $\und{h}$ which do not generate local $C^2$ solutions, a result analogous to Lebeau's
theorem for Kelvin-Helmholtz \cite{lebeau2002regularite}. The proof of \cite{lerner2010instability} relies strongly on a representation of solutions based on the method of characteristics, specific to scalar equations, which was developed earlier in \cite{metivier1985uniqueness}.

%In \cite{lerner2010instability}, Lerner, Morimoto and Xu extend the Theorem 2.1 of \cite{metivier2005remarks} to the case of complex scalar equations, a prototypical example being 
%\be
	%\label{burgers}
	%\d_t u + u\d_{x}u = i
%\ee
%
%\noindent with real data. In \cite{lerner2010instability} is studied the stability of Cauchy-Kovalevaskaya solutions with respect to non-analytical perturbations of the initial datum. Under the assumption of a weak defect of hyperbolicity at a point $(x_0,u_0)\in\R^{d}\times\C$, they show that for any analytical initial datum $\und{h}$ such that $\und{h}(x_0) = u_0$ there exists a smooth initial datum $h$ $C^{\infty}$-close to $\und{h}$ such that the scalar equation has no $C^{2}$ solution. By instance, the Burgers equation \eqref{burgers} is initially hyperbolic as the initial datum is real, but the imaginary source term forces the spectrum of the operator to leave the real axis as soon as $t>0$. Their proof relies strongly on a representation of the solutions based on the method of characteristics, specific to scalar equations, which was developed earlier in \cite{metivier1985uniqueness}. The results of instability of \cite{lerner2010instability} are stronger than ours, but the analysis does not carry over to systems.

%%%%%%%%%%%%%%%%%%%%%%%%%%%%%%%%%%%%%%%%%%%%%%%%%%%%%%%%%%%%%%%%%%%%%%%%%%%%%%%%%%%%%%%%%%%%%%%%%%%
\subsubsection{Lerner, Nguyen and Texier's result on transition to ellipticity for general systems}
%%%%%%%%%%%%%%%%%%%%%%%%%%%%%%%%%%%%%%%%%%%%%%%%%%%%%%%%%%%%%%%%%%%%%%%%%%%%%%%%%%%%%%%%%%%%%%%%%%%

In \cite{lerner2015onset}, Lerner, Nguyen and Texier extend the analysis of \cite{lerner2010instability} to systems \eqref{Cauchy}. The result of \cite{lerner2015onset} shows an instantaneous lack of Hölder well-posedness of the flow, with an arbitrarily large loss of derivatives, under appropriate assumptions of transition to ellipticity. The analysis of \cite{lerner2015onset} is based on the method of approximation of pseudo-differential flows introduced in \cite{texier2014approximations}. One key observation in \cite{lerner2015onset} is that for systems, many types of transitions may occur. The focus in \cite{lerner2015onset} is on genuinely non-scalar transitions (more about this specific point in Remark \ref{remark.deux}). For these, the propagator generically grows in time like the Airy function.

%In \cite{lerner2015onset} by Lerner, Nguyen and Texier the focus is on general systems \eqref{Cauchy}. The main result of LNT states a lack of Hölder continuity for the flow, just like our main results. Our results are more precise, in two distinct directions: first, we disprove regularity of the flow in Gevrey spaces, which are included in all Sobolev spaces; just like in LNT this lack of regularity takes the form of a deviation from a given reference solution, which may be of Cauchy-Kovalevskaya type. Second, we prove existence of solutions up to the observation time at which the deviation from the reference solution is recorded. We introduce appropriate fast scales and use the implicit
%representation \eqref{intro.pointfixe}, where the propagator has constant coefficients. By contrast, the proof of LNT is based on the method of approximations of pseudo-differential flows introduced in \cite{texier2014approximations}.

%%%%%%%%%%%%%%%%%%%%%%%%%%%%%%%%%%%%%%%%%%%%%%%%%%%%%%%%%%
\subsubsection{Defect of hyperbolicity in Maxwell systems}
%%%%%%%%%%%%%%%%%%%%%%%%%%%%%%%%%%%%%%%%%%%%%%%%%%%%%%%%%%

There is a strong analogy between the progression from \cite{metivier2005remarks} to our present results and recent results \cite{lu2015}, \cite{lu2016resonances} in geometric optics. In \cite{lu2015}, Lu and Texier study large-amplitude solutions to Maxwell-based systems in the small wavelength limit. They show that in appropriate coordinates, resonances in frequency correspond to points of weak hyperbolicity. Thus at the resonances, the subprincipal symbol plays a role in the stability analysis. Under a Levi condition, hyperbolicity is violated around the resonances, and WKB solutions do not approximate exact solutions issued from appropriate nearby initial data, no matter how precise the order of the WKB approximation. This result is somehow analogous to Métivier's initial ellipticity result. Following \cite{lu2015}, Lu studied in \cite{lu2016resonances} a situation in which WKB solutions are destabilized by terms which are not present in the initial data. That is, the Levi condition of \cite{lu2015} is satisfied initially, but higher-order harmonics of the WKB solutions, which are generated by the nonlinearities in the course of the propagation, are associated with higher-order resonances. For these resonances, the Levi conditions may not be satisfied, leading to instability. This framework is somehow similar to ours, with an instability which develops in time, starting from an initially hyperbolic situation.

%%%%%%%%%%%%%%%%%%%%%%%%%%%%%%%%%%%%%%%%%%%%%%%%%%%%%%%%%%%%%%%%%%%%%%%%%%%%%%%%%
\subsubsection{Ill-posedness results for hydrostatic Euler and related equations}
%%%%%%%%%%%%%%%%%%%%%%%%%%%%%%%%%%%%%%%%%%%%%%%%%%%%%%%%%%%%%%%%%%%%%%%%%%%%%%%%%

In \cite{MR3509003}, Han-Kwan and Nguyen study the hydrostatic Euler and some singular Vlasov equations through the point of view of an abstract PDE $\d_t U - \mathcal{L}U = \mathcal{Q}(u,u)$, where $\mathcal{L}$ and $\mathcal{Q}$ are (bi)linear non-local operators. For each of those equations, the corresponding operator $\mathcal{L}$ exhibits an unbounded unstable spectrum, similar to the notion of ellipticity used in the present paper. Inspired notably by the work \cite{metivier2005remarks} of Métivier, Han-Kwan and Nguyen develop an analytical framework in order to prove an Hadamard instability result in all Sobolev spaces. One main ingredient is, as in \cite{metivier2005remarks} and the present work, to prove the existence of a family of analytical solutions for the abstract PDE which carry the expected exponential growth in time (called \textit{loss of analyticity for the semi-group} in \cite{MR3509003}). 

%There is an analogy between resonances in geometric optics and defect of hyperbolicity. In \cite{lu2015}, Yong Lu and Benjamin Texier study the role of resonances in the instability of high-frequency WKB solutions for Maxwell-type systems of the form
%\be
	%\label{maxwell}
	%\d_t U + \frac{1}{\e}A_0 U + \sum_j A_j \d_{x_j}U = \frac{1}{\sqrt{\e}} B(U,U)
%\ee
%
%\noindent where $\e$ is a small parameter that represents the wavelength of the light, and $A_0$ is skew-symmetric and the $A_j$ symmetric.  One of the main results of \cite{lu2015} is to show the existence of a Levy compatibility condition that suffices to distinguish between stability and instability zones. In both this case and ours, the problem is initially hyperbolic as $A_0$ is skew-symmetric, but the instability comes here from resonances due to the nonlinearity $B$ and the large prefactor $1/\sqrt{\e}$ in front of it.

%%%%%%%%%%%%%%%%%%%%%%%%%%%%%%%%%%%%
\subsection{Overview of the paper}
%%%%%%%%%%%%%%%%%%%%%%%%%%%%%%%%%%%%

Our assumptions are based on the framework set out in \cite{lerner2015onset}, the results of which we extend in two distinct ways: we prove existence of solutions up to the observation time at which the Hadamard instability is recorded, and we measure the deviation in Gevrey spaces.

We assume that for a specific frequency $\xi_0\in\R^{d}$ the linear part of the principal symbol at $u=\vec{u}_0\in\R^{N}$ defined by \eqref{intro.def.und.A} has a real spectrum at time $t=0$ while non real eigenvalues appear for $t>0$. In this sense the operator experiences a transition from initial hyperbolicity ($t = 0$, real eigenvalues) to eventual ellipticity ($t > 0$, non-real eigenvalues).

A sharp difference with the initially elliptic case lies in the normal forms of the operators. Indeed, the elliptic case is reducible to the case where $\und{A}$ is a triangular matrix with non real and conjugated diagonal entries. 

By contrast, transitions in time appear in many ways. There is not one single normal form. Section \ref{section.assumptions} will be devoted to the descriptions of such transitions in time and the associated normal forms for systems of size $N=2$. In particular, this paper focuses on two particular normal forms, described in the next paragraphs.

%%%%%%%%%%%%%%%%%%%%%%%%%%%%%%%%%%%%%%%%%%%%%%%%
\subsubsection{The smoothly diagonalizable case}
\label{subsubsection.smooth}
%%%%%%%%%%%%%%%%%%%%%%%%%%%%%%%%%%%%%%%%%%%%%%%%

Under Assumptions {\rm \ref{hypo.branching}}, {\rm \ref{hypo.smooth}} and {\rm \ref{hypo.semi-simple}} (see Proposition \ref{prop.normalform.smooth} below), there holds
$$ 
	\und{A}(t,x) \approx \und{A}^{{\rm S}}(t) :=
	\begin{pmatrix}
		0 & t \\
		- \g_0^2 t & 0
	\end{pmatrix}
$$

\noindent with $\g_0 > 0$. Here $\approx$ means equality up to higher order terms in the Taylor expansion in time and space, and up to a change of basis. The matrix $\und{A}^{{\rm S}}(t)$ is smoothly diagonalisable in $\C$, with smooth eigenvalues $ \pm i \g_0 t $. This case is mostly scalar ; it is analogous to a degenerate Cauchy-Riemann problem. 

Our analysis shows that our method in \cite{morisse2016I} is robust enough to allow for such a weak defect of hyperbolicity. We replace ansatz $(3.2)$ therein by $u_{\e}(t,x) = \u(t /\e^{1/2} , x , (x-x_0)\cdot \xi_0/\e)$. For such $\und{A}^{{\rm S}}(t)$, the growth for the associated propagator solving
$$
	\d_{s} U^{{\rm S}}(s',s,\theta) - \und{A}^{{\rm S}}( s) \d_{\theta}U^{{\rm S}}(s',s,\theta)  = 0
$$

\noindent is like
\be
	\label{intro.growth.smooth}
	|U_n^{{\rm S}}(s',s)| \lesssim \exp\left(\int_{s'}^{s} \g_{{\rm S}}^{\sharp}(\tau) d\tau \right) \; , \quad \forall \, 0 \leq s' \leq s \; , \; \forall\,n\in\Z^{d}
\ee

\noindent for the Fourier coefficients of $U^{{\rm S}}(s',s,\theta)$, with $\g_{{\rm S}}^{\sharp}(\tau) = \g_0 \tau$.
	
%%%%%%%%%%%%%%%%%%%%%%%%%%%%%%%%%%%%%%%%
\subsubsection{The degenerate Airy case}
\label{subsubsection.airy}
%%%%%%%%%%%%%%%%%%%%%%%%%%%%%%%%%%%%%%%%

Under Assumptions {\rm \ref{hypo.branching}} and {\rm \ref{hypo.degenerate}} (see Proposition \ref{prop.normalform.stiff} below), there holds
$$
	\und{A}(t,x) \approx \und{A}^{{\rm Ai}}(t,x) :=
	\begin{pmatrix}
		0 & 1 \\
		-\g_0^2(t-t_{\star}(x)) & 0 
	\end{pmatrix}
$$

\noindent where $\approx$ means equality up to higher order terms in the Taylor expansion in time and space and $t_{\star}(x) \geq 0$ in a whole neighborhood of $x=x_0$\footnote{If there were $x_1$ such that $t_{\star}(x_1) <0$, we would be in the case of initial ellipticity and Métivier's result would apply.}. 

The time transition function $t_{\star}(x)$ defines the boundary between the elliptic and hyperbolic zones. Indeed, for $t < t_{\star}(x)$, the eigenvalues are $\pm \sqrt{t_{\star}(x)-t}$ while for $t > t_{\star}(x)$ the eigenvalues are $	\pm i\sqrt{t-t_{\star}(x)} $. 

The transition between hyperbolicity and ellipticity is thus not uniform in space, and depends on the space-dependent transition time $t_{\star}(x)$. In order to use and develop the method of \cite{morisse2016I}, we have to treat the transition time as a remainder term and verify its smallness in the framework. From that view, the non degenerate case $ t_{\star}(x) = O((x-x_0)^2)$ is out of reach of the method presented in this paper, and requires special attention - we devoted two companion papers \cite{morisse2016lemma} and \cite{morisse2016IIz} to the subject - more about this specific point in Remark \ref{remark.deux} below. We will focus here on the degenerate case
$$ 
	t_{\star}(x) = O((x-x_0)^4) .
$$

\noindent Note that the cases of odd power of $x$ are in contradiction with the assumption of non-negativity of $t_{\star}$ around $x=x_0$.

We emphasize also the fact that the eigenvalues of $\und{A}^{{\rm Ai}}$ are $C^0$ in time but not $C^1$, hence the stiffness of this case.

In this framework, we replace ansatz $(3.2)$ of \cite{morisse2016I} by $u_{\e}(t,x) = \u(t /\e^{2/3} , x , (x-x_0)\cdot \xi_0/\e)$. As such a transition is not semi-simple, the previous ansatz induces the following equation for the propagator
$$
	\d_{s} U^{{\rm Ai}}(s',s,\theta) - \e^{-1/3}\und{A}^{{\rm Ai}}(s,x_0) \d_{\theta} U^{{\rm S}}(s',s,\theta) = 0
$$

\noindent which as then a growth like
\be
	\label{intro.growth.airy}
	|U_n^{{\rm Ai}}(s',s)| \lesssim \e^{-1/3} \exp\left( \int_{s'}^{s} \g_{{\rm Ai}}^{\sharp}(\tau) d\tau \right) \; , \quad \forall \, 0 \leq s' \leq s \; , \; \forall\,n\in\Z^{d}
\ee

\noindent for the Fourier coefficients of $U^{{\rm Ai}}(s',s,\theta)$ with $\g_{{\rm Ai}}^{\sharp}(\tau) = \g_0 \tau^{1/2}$ which is typical of the Airy growth.

\begin{remark}
	\label{remark.deux}
	In {\rm \cite{lerner2015onset} }, the authors allow for generic non-scalar transitions, for which $t_{\star}(x) = O((x-x_0)^2) $. In particular, the space-time domain $\{(t,x) : (x-x_0)^2 \leq t\}$ is included in the domain of hyperbolicity. As we will see precisely in the course of the proof of Proposition {\rm \ref{prop.below.airy}}, in our context this space-time domain is too large for the standard Cauchy-Kovalevskaya theorem to apply. Thus, in the case $t_{\star}(x) = O((x-x_0)^2)$, we need a specific Gevrey well-posedness result in that space-time domain before observing the instability develop in the elliptic domain. This Gevrey well-posedness result is the object of the article {\rm \cite{morisse2016lemma}}, and the completion of the instability proof in the case $t_{\star}(x) = O((x-x_0)^2) $ is the object of the article {\rm \cite{morisse2016IIz}}.

\end{remark}

%%%%%%%%%%%%%%%%%%%%%%%%%%%%%%%%%%%%%%%%%%%%%%%%%%%%%%%%%%%%%%%%%%%%%%%%%%%
\subsubsection{Example: compressible Euler with Van der Waals pressure law}
%%%%%%%%%%%%%%%%%%%%%%%%%%%%%%%%%%%%%%%%%%%%%%%%%%%%%%%%%%%%%%%%%%%%%%%%%%%

Transitions of the principal symbol from hyperbolicity to ellipticity, as described in the above paragraphs, are observed in physical equations describing \textit{phase transitions}. One such system (mentioned in both \cite{metivier2005remarks} and \cite{lerner2015onset}) is the compressible Euler equations in one spatial dimension, with a Van der Waals pressure law:
\be
	\label{vdw} 
	\begin{cases}
		& \d_t u_1 + \d_{x}u_2 = 0  \\
		& \d_{t}u_2 + \d_{x}(p(u_1)) = 0
	\end{cases} 
\ee

\noindent where $p$ follows a Van der Waals equation of state, for which there holds $p'(u_1) \leq 0$, for some $u_1\in\R$. The system is hyperbolic (resp. elliptic) for $p'(u_1) >0$ (resp. for $p'(u_1) <0$). For solutions which leave the hyperbolic zone, a phase transition occurs. This corresponds for us to the catastrophic growth recorded in the elliptic zone. If for instance the elliptic zone is defined by $\{|u_1|\leq \delta\}$, for some $\delta>0$, then solutions may enter the elliptic zone only to leave it immediately, due to the exponential growth.

%%%%%%%%%%%%%%%%%%%%%%%%%%%%%%%%%%%%%%%%%%%%%%%%%%%%%%%%%%%%%%%%%%%%%%%%%%%%%%%
%%%%%%%%%%%%%%%%%%%%%%%%%%%%%%%%%%%%%%%%%%%%%%%%%%%%%%%%%%%%%%%%%%%%%%%%%%%%%%%

\section{Main assumptions and results}
\label{section.assumptions}

%%%%%%%%%%%%%%%%%%%%%%%%%%%%%%%%%%%%%%%%%%%%%%%%%%%%%%%%%%%%%%%%%%%%%%%%%%%%%%%

\subsection{Branching eigenvalues and defect of hyperbolicity}

%%%%%%%%%%%%%%%%%%%%%%%%%%%%%%%%%%%%%%%%%%%%%%%%%%%%%%%%%%%%%%%%%%%%%%%%%%%%%%%

We look at the possible cases of a defect of hyperbolicity, that is transitions from initial hyperbolicity to ellipticity at time $t>0$, following the work of Lerner, Nguyen and Texier of \cite{lerner2015onset}. 

We introduce first
\be
	\label{def.A}
	A(t,x,u) = \sum_{j} A_j(t,x,u) \xi_{0,j}.
\ee

\noindent We assume there are $(x_0,\vec{u}_0,\xi_0) \in\R^{d}\times\R^{N}\times\R^{d}$ and $r_0>0$ such that the principal symbol defined by
\be
	\label{def.A_0}
	\und{A}(t,x) = A(t,x,\vec{u}_0)
\ee

\noindent satisfies
\be
	\label{spectrum.transition.a}
		{\rm Sp}\left(\und{A}(0,x)\right) \subseteq \R \;, \quad \forall x\in B_{r_0}(x_0) 
\ee

\noindent which stands for initial and local hyperbolicity around $x_0\in\R^{d}$. Note that, as soon as there is some $x_1\in\R^{d}$ such that $\und{A}(0,x_1)$ has non real spectrum, we are in the case of initial ellipticity treated in \cite{morisse2016I}.

We assume also that, for small times $t>0$, there are some $x$ close to $x_0$ such that
\be
	\label{spectrum.transition.b}
	{\rm Sp}\left(\und{A}(t,x)\right) \not\subseteq \R \quad , \quad \forall \,t>0 .
\ee

\noindent Condition \eqref{spectrum.transition.a} stands for initial and local hyperbolicity around $x_0$ ; condition \eqref{spectrum.transition.b} expresses the ellipticity of $\und{A}$ at time $t>0$. Up to translations in $x$ and $u$, which do not affect our forthcoming assumptions, and by homogeneity in $\xi$, we may assume
\be
	\label{etc}
	x_0 = 0 \quad , \quad \vec{u}_0 = 0 \quad , \quad |\xi_0| = 1 .
\ee

Since the $A_j$ have real coefficients, non-real eigenvalues of $\und{A}(t,x)$ appear in conjugate pairs. For such a pair $\lambda_{\pm}(t,x)$, by reality of the eigenvalues at $t=0$ we have a double eigenvalue $\lambda_{-}(0,x) = \lambda_{+}(0,x) \in\R$ of $\und{A}(0,x)$. To avoid higher order transitions (which would involve eigenvalues of multiplicity $3$ or greater), we assume the eigenvalues of $\und{A}(0,0)$ to be distinct and simple, except for one double eigenvalue:

\begin{hypo}
	\label{hypo.1}
	We assume the eigenvalues of $\und{A}(0,0)$ to be distinct and simple, except for one double eigenvalue.
\end{hypo}

We block diagonalize the principal symbol into $\und{A}(t,x)^{(0)}$ and $\und{A}(t,x)^{(1)}$. The block $\und{A}(t,x)^{(0)}$ is a $2\times2$ matrix corresponding to the double eigenvalue, and the $(N-2)\times(N-2)$ block $\und{A}^{(1)}$ has simple real eigenvalues at $t=0$ in a whole neighborhood of $x=0$. Thanks to Assumption \ref{hypo.1} the block diagonalization is smooth. Therefore we focus our discussion on $\und{A}^{(0)}$, and we may assume $N=2$, that is $\und{A} \equiv \und{A}^{(0)}$. 

The question is now to describe the possible matrices $\und{A}(t,x)$ satisfying conditions \eqref{spectrum.transition.a} and \eqref{spectrum.transition.b}. Following \cite{lerner2015onset}, we reformulate the conditions \eqref{spectrum.transition.a} and \eqref{spectrum.transition.b} in terms of the characteristic polynomial of $\und{A}$ defined as
\be
	\label{defi.poly.char}
	P(\lambda,t,x) = {\rm det}\left(\lambda - \und{A}(t,x)\right)
\ee

\noindent which is simply in the case $N=2$
\be
	\label{poly.char.calc}
	P(\lambda,t,x) = \left(\lambda - \frac{1}{2}{\rm Tr}\und{A}(t,x)\right)^2 + \Delta(t,x) 
\ee

\noindent where we define
\be
	\label{def.Delta}
	\Delta(t,x) = {\rm det} \, \und{A}(t,x) - \left(\frac{1}{2}{\rm Tr}\und{A}(t,x)\right)^2 .
\ee

\noindent Thus the real or complex nature of the spectrum depends on the sign of $\Delta$. So condition \eqref{spectrum.transition.a} is equivalent in terms of $\Delta$ to
\be
	\label{Delta.initial.hyp}
	\Delta(0,x) \leq 0 \quad , \quad \forall \, x\in B_{r_0}(0) .
\ee

\noindent As a double eigenvalue $\lambda_{-}(0,x) = \lambda_{+}(0,x) \in\R$ of $\und{A}(0,x)$ corresponds to a double root of $P(\lambda,0,x)$, we formulate the following Assumption:

\begin{hypo}[Branching eigenvalues]
	\label{hypo.branching}
	In addition to {\rm \eqref{Delta.initial.hyp}}, we assume that there exists some $\lambda_0\in\R$ such that
	\be
		\label{double.root}
		P(\lambda_0,0,0) = 0 \quad , \quad \d_{\lambda}P(\lambda_0,0,0) = 0 .
	\ee
\end{hypo}

\begin{remark}
	Note that condition \eqref{double.root} is equivalent to
	$$ 
		\lambda_0 = \frac{1}{2}{\rm Tr}\und{A}(0,0) \quad , \quad \Delta(0,0)=0 .
	$$
\end{remark}

For condition \eqref{spectrum.transition.b} to be satisfied, that is for a conjugate pair of eigenvalues to appear as $t>0$, $\Delta(t,x)$ has to be positive for $t>0$. The eigenvalues of $\und{A}$, which are the zeroes of $P$, are then expressed by the square roots of $\Delta$. Even though the regularity of $\Delta$, being an algebraic combination of the coefficients of $\und{A}$, is analytic, the regularity of the square roots of $\Delta$ can be of course much weaker.  How much rougher than $\Delta$ may $\sqrt{\Delta}$ be has been studied in particular by Glaeser \cite{glaeser1963racine}. The question of the regularity of the eigenvalues and of the eigenvectors is here of importance as we work in the analytic framework:  we may not use non-smooth (in time and space) changes of basis, since the methods we use, following \cite{metivier2005remarks}, strongly rely on analyticity. In particular, we may not diagonalize the principal symbol if the eigenvectors are not smooth.

%performing a rough change of basis (depending on time and space) would imply a high loss of regularity, whereas the methods used in this paper (following the ones used in \cite{metivier2005remarks}) are strongly based on the analyticity of the coefficients of the equation.

%%%%%%%%%%%%%%%%%%%%%%%%%%%%%%%%%%%%%%%%%%%%%%%%%%%%%%%%%%%%%%%%%%%%%%%%%%%%%%%

\subsection{The case of a smooth transition}
\label{subsection.smooth.assum}

%%%%%%%%%%%%%%%%%%%%%%%%%%%%%%%%%%%%%%%%%%%%%%%%%%%%%%%%%%%%%%%%%%%%%%%%%%%%%%%

%%%%%%%%%%%%%%%%%%%%%%%%%%%%%%%%%%%%%%%%%%
%\subsubsection{Reduction to a normal form}
%%%%%%%%%%%%%%%%%%%%%%%%%%%%%%%%%%%%%%%%%%

For the square roots of $\Delta$ to be as smooth as $\Delta$, the discriminant $\Delta$ has to be the square of a smooth function $\delta(t,x)$:
$$ 
	\Delta(t,x) = \delta(t,x)^2 .
$$

\noindent In this case, note that $\Delta(0,x) = \delta(0,x)^2 \geq 0$. Since we assume also that $\Delta(0,x) \leq 0$ by \eqref{Delta.initial.hyp}, we get 
$$ \delta(0,x) = 0 \quad , \quad \forall x\in B_{r_0}(0) . $$

\noindent This is equivalent to the existence of some analytic function $\tilde{\delta}(t,x)$ such that 
$$ 
	\delta(t,x) = t\tilde{\delta}(t,x) .
$$ 

\noindent We sum up all this in the following

\begin{hypo}[Smooth transition]
	\label{hypo.smooth}
	There is a function $\delta(t,x)$ analytic in the $t$ and $x$ variables such that
	\be
		\label{smooth.a}
		\Delta(t,x) = \left(t\delta(t,x)\right)^2
	\ee
	\noindent with 
	\be
		\label{smooth.b}
		\delta(0,0) = \g_0 > 0 .
	\ee
\end{hypo}

Under Assumption \ref{hypo.smooth}, since $\Delta(0,x) \equiv 0$ the eigenvalues of $\und{A}(0,x)$ are the double eigenvalue $ \frac{1}{2}{\rm Tr}\und{A}(0,x)$. There are two cases\footnote{As opposed to the case of a stiff transition, described in Section \ref{subsection.stiff}, where $\und{A}(0,0)$ is not semi-simple.}, as $\und{A}(0,x)$ could be semi-simple or not. In what follows we add the assumption 

\begin{hypo}[Semi-simplicity]
	\label{hypo.semi-simple}
	The unique eigenvalue of $\und{A}(0,0)$ is semi-simple, for all $x$ near $x=0$.
\end{hypo}

\noindent This assumption is Hypothesis 1.5 in \cite{lerner2015onset}. We can now prove the following normal form result

\begin{prop}[Normal form for the smooth transition]
	\label{prop.normalform.smooth}
	Under Assumptions {\rm \ref{hypo.branching}}, {\rm \ref{hypo.smooth}} and {\rm \ref{hypo.semi-simple}}, there is an analytical change of basis $Q_0(t,x)\in\R^{2\times2}$ such that
	\be
		\label{normalform.smooth}
		Q_0^{-1}(t,x)\left(\und{A}(t,x) - \frac{1}{2}{\rm Tr}\und{A}(t,x)\,{\rm Id}\right)Q_0(t,x) = 
		\begin{pmatrix}
			0 & t \\
			-t\delta^2 & 0
		\end{pmatrix} .
	\ee

\end{prop}

\begin{proof}

	We denote
	$$
		\begin{pmatrix}
			a_{11} & a_{12} \\
			a_{21} & -a_{11}
		\end{pmatrix}
		= \und{A}(t,x) - \frac{1}{2}{\rm Tr}\und{A}(t,x) {\rm Id} .
	$$

	\noindent By definition \eqref{def.Delta} there holds $	\Delta = -a_{11}^2 - a_{12}a_{21} $ and then, by \eqref{smooth.a} in Assumption \ref{hypo.smooth}
	\be
		\label{determinant}
		-a_{11}^2 - a_{12}a_{21} = t^2 \delta(t,x)^2  .
	\ee

	\noindent By Assumption \ref{hypo.semi-simple} the matrix $ (a_{ij})_{i,j} $ satisfies
	$$ 
		a_{11}(0,x) = a_{12}(0,x) = a_{21}(0,x) = 0 
	$$

	\noindent so that there are smooth functions $\widetilde{a_{ij}}(t,x)$ such that $ a_{ij}(t,x) = t\widetilde{a_{ij}}(t,x) $. Hence by \eqref{determinant} we get 
	$$
		\delta(t,x)^2 = -\widetilde{a_{11}}^2 - \widetilde{a_{12}}\widetilde{a_{21}} .
	$$
	
	\noindent As $\delta(0,0)^2 >0$ by Assumption \ref{hypo.smooth} (2), the term $\widetilde{a_{12}}\widetilde{a_{21}}(0,0)$ is non zero. Hence either one of $\widetilde{a_{12}}(0,0)$ or $\widetilde{a_{21}}(0,0)$ is non zero. In the first case, the matrix
	$$
		Q_0(t,x) = 
		\begin{pmatrix}
			\widetilde{a_{11}} & 1 \\
			\widetilde{a_{21}} & 0
		\end{pmatrix}
	$$

	\noindent is such that \eqref{normalform.smooth} holds. The second case is treated in the same way, which suffices to end the proof.
	
\end{proof}

%%%%%%%%%%%%%%%%%%%%%%%%%%%%%%%%%%%%%%%%%%%%%%%%%%%%%%%%%%%%%%%%%%%%%%%%%%%%%%%

\subsection{The case of a stiff transition}
\label{subsection.stiff}
%%%%%%%%%%%%%%%%%%%%%%%%%%%%%%%%%%%%%%%%%%%%%%%%%%%%%%%%%%%%%%%%%%%%%%%%%%%%%%%

%%%%%%%%%%%%%%%%%%%%%%%%%%%%%%%%%%%%%%%%%%%%%%%%%%%%%%%%%%%%%%%%%%%%%%%%
%\subsubsection{Reduction to a normal form under a degeneracy assumption}
%%%%%%%%%%%%%%%%%%%%%%%%%%%%%%%%%%%%%%%%%%%%%%%%%%%%%%%%%%%%%%%%%%%%%%%%

If $\Delta$ is not the square of a function, its square roots are typically not as smooth as $\Delta$. In fact, for any $k\in\N$ it is possible to find $\Delta$ such that it is analytic, but its square roots are $C^{k}$ and not $C^{k+1}$. The first non degenerate case of this kind is when
\be
	\label{dt.Delta}
	\d_{t}\Delta(0,0) >0 
\ee

\noindent which implies that $\Delta(t,0)^{1/2} \sim t^{1/2} $ which is $C^{0}$ but not $C^{1}$ at $t=0$. With $\Delta(0,0) = 0$ by Assumption \ref{hypo.branching}, condition \eqref{dt.Delta} and the implicit function theorem give the existence of an analytic function $t_{\star}(x)$ such that
\be
	\label{def.t_star}
	\Delta(t,x) = 0 \Longleftrightarrow t = t_{\star}(x) \quad \text{locally around } (t,x)=(0,0) .
\ee

\noindent Introducing
$$
	e(t,x) = \int_{0}^{1}\d_{t}\Delta((1-\tau)t_{\star}(x) + \tau t,x) d\tau 
$$

\noindent there holds
\be
	\label{equality.Delta.tstar}
	\Delta(t,x) = \left(t - t_{\star}(x)\right)e(t,x) .
\ee

\noindent As $\Delta$ is analytic, $e$ is also analytic, and satisfies
$$
	e(0,0) = \d_{t}\Delta(0,0) >0
$$

\noindent so that $e$ is positive around $(0,0)$. Then the sign of $\Delta(t,x)$, hence the real or complex
nature of the spectrum of $\und{A}(t,x)$, is given by the sign of $t - t_{\star}(x)$, a situation comparable to the one described in Section 1.2.3 of \cite{lerner2015onset}: 
\begin{itemize}
	\item For $(t,x)$ under the transition curve $\{(t_{\star}(x),x)\}$ the eigenvalues of $\und{A}(t,x)$ are real.
	\item For $(t,x)$ above the transition curve, the eigenvalues of $\und{A}(t,x)$ have a non-zero imaginary part like $\pm i (t-t_{\star}(x))^{1/2}$.
\end{itemize}

\noindent The question is then to describe $t_{\star}$. First, as $\Delta(0,0) = 0$, 
\be
	\label{zero.t_star}
	t_{\star}(0) = 0.
\ee

\noindent As $\Delta(0,x) \leq 0$ for $x\in B_{r_0}(0)$, we have
$$
	t_{\star}(x) \geq 0 \quad , \quad \forall x\in B_{r_0}(0) 
$$

\noindent which implies
\be
	\label{zero.dx.t_star}
	\d_{x} t_{\star}(0) = 0 
\ee

\noindent so that the Taylor expansion of $t_{\star}(\cdot)$ around $x=0$ is as
$$
	t_{\star}(x) = \frac{1}{2} \sum_{j,k}\d_{x_j}\d_{x_k} t_{\star}(0) \,x_jx_k + O(x^3)
$$

\noindent and the Hessian $\left( \d_{x_j}\d_{x_k} t_{\star}(0) \right)_{j,k} $ is a nonnegative matrix. But as we will see in the course of the proof of Proposition \ref{prop.below.airy}, the non degenerate case $\left( \d_{x_j}\d_{x_k} t_{\star}(0) \right)_{j,k} \neq 0$ cannot be dealt with our method. We then assume
\be
	\d_{x_j}\d_{x_k} t_{\star}(0) = 0 \quad , \quad \forall \, j,k = 1 , \ldots , d .
\ee

\noindent Just as before, inequality \eqref{Delta.initial.hyp} implies that third order derivatives of $t_{\star}(\cdot)$ are null at $x=0$, and there holds
$$
	t_{\star}(x) = O(x^4) .
$$

In order to sum up those assumptions in a more intrinsic way, we express derivatives of $t_{\star}$ by derivatives of $\Delta$. By definition \eqref{def.t_star} of $t_{\star}$, there holds $	\Delta(t_{\star}(x),x) = 0 $ hence, differentiating with respect to $x$ and taking $x=0$:
$$
	\d_{x}t_{\star}(0) \,\d_{t}\Delta(0,0) + \d_{x}\Delta(0,0) = 0.
$$

\noindent As $\d_{t}\Delta(0,0) > 0$, equality \eqref{zero.dx.t_star} is then equivalent to
$$
	\d_{x}\Delta(0,0) = 0.
$$

\noindent By Faà di Bruno formula on iterate derivatives applied to the equality $	\Delta(t_{\star}(x),x) = 0 $, we may prove by induction that $t_{\star}(x) = O(x^4)$ is equivalent to the following

\begin{hypo}[Degenerate stiff transition]
	\label{hypo.degenerate}
	We assume 
	$$
		\d_{x}^{\a}\,\Delta(0,0) = 0 \quad , \quad \forall \,\a\in\N^{d} \text{ with } |\a| \leq 3 .
	$$
\end{hypo}

We prove now a normal form expression for $\und{A}(t,x)$:

\begin{prop}[Normal form for the stiff transition]
	\label{prop.normalform.stiff}
	Under Assumptions {\rm \ref{hypo.branching}} and {\rm \ref{hypo.degenerate}}, there are an analytical change of basis $Q_0(t,x)$ and real analytical functions $t_{\star}(x)$ and $e(t,x)$ such that
	\be
		\label{normalform.stiff}
		Q_0^{-1}\left(\und{A}(t,x) - \frac{1}{2}{\rm Tr}\und{A}(t,x)\,{\rm Id}\right)Q_0 = 
		\begin{pmatrix}
			0 & 1 \\
			-(t-t_{\star})e & 0
		\end{pmatrix} .
	\ee
\end{prop}

\begin{proof}

	\noindent By definition of $\Delta$ and denoting
	$$
		\und{A} - \frac{1}{2}{\rm Tr}\und{A} \,{\rm Id} = 
		\begin{pmatrix}
			a_{11} & a_{12} \\
			a_{21} & -a_{11}
		\end{pmatrix}
	$$

	\noindent we get $\Delta = -a_{11}^2 - a_{12}a_{21}$. As $\Delta \sim t$ both $a_{12}$ and $a_{21}$ cannot both be zero at $(0,0)$. Assuming that $a_{21}(0,0) \neq 0$, the matrix
	$$
		\label{Q.stiff}
		Q_0(t,x) =
		\begin{pmatrix}
			a_{11} & 1 \\
			a_{21} & 0
		\end{pmatrix}
	$$

	\noindent is an analytical change of basis such that \eqref{normalform.stiff} holds.

\end{proof}

\begin{remark}
	\label{remark.2}
	Note that, on the contrary of the normal form of the smooth transition given in Proposition {\rm \ref{prop.normalform.smooth}}, the normal form of the stiff transition is not semi-simple. This is of importance, as non-semisimplicity introduces an additional factor $\e^{-1/3}$ in the upper bound \eqref{growth.propa.airy} of the Airy propagator, to be compared with the upper bound \eqref{growth.propa.smooth} in the smooth case.
\end{remark}

We add the following assumption:

\begin{hypo}[Genuinely nonlinear zeroth-order perturbation]
	\label{hypo.structural}
	We assume that $f(t,x,u)$ is quadratic in $u$ locally around $u=\vec{u}_0$, that is
	$$
		\d_{u} f(t,x,u) \big|_{u=\vec{u}_0} \equiv 0
	$$
	
	\noindent in a neighborhood of $(t,x) = (0,0)$.
\end{hypo}

%%%%%%%%%%%%%%%%%%%%%%%%%%%%%%%%%%%%%%%%%%%%%%%%%%%%%%%%%%%%%%%%%%%%%%%%%%%%%%%

\subsection{Statement of the results}
\label{subsection.results}

%%%%%%%%%%%%%%%%%%%%%%%%%%%%%%%%%%%%%%%%%%%%%%%%%%%%%%%%%%%%%%%%%%%%%%%%%%%%%%%

We recall first the definition of conical domain of $\R_{t}\times\R_{x}^{d}$ centered at $(t,x) = (0,0)$, as in Definition $2.2$ in \cite{morisse2016I}. We denote
\be
	\label{defi.Omega}
	\Omega_{R,\rho} = \bigcup_{t\geq 0} \{t\}\times\Omega_{R,\rho,t} = \left\{(t,x) \in\R\times\R^{d} \;\Big|\; 0\leq t < \rho^{-1} ,\; R |x|_{1} + \rho t <1 \right\} .
\ee

%For the Definition of Hölder well-posedness in Gevrey spaces, see Definition 2.3 in \cite{morisse2016I}.

\begin{theo}[Gevrey ill-posedness of the smooth case]
	\label{theorem.smooth}
	Under Assumptions {\rm \ref{hypo.1}}, {\rm \ref{hypo.branching}}, {\rm \ref{hypo.smooth}}, {\rm \ref{hypo.semi-simple}} and {\rm \ref{hypo.structural}}, the Cauchy problem \eqref{Cauchy} is not Hölder well-posed in Gevrey spaces $G^{\sigma}$ for all $\sigma \in(0,\sigma_0)$ with
	$$ \sigma_0 = 1/3 . $$
	
	\noindent That is for all $c>0$, $K$ compact of $\R^{d}$ and $\a\in(0,1]$, there are sequences $R_{\e}^{-1} \to 0$ and $\rho_{\e}^{-1} \to 0$, a family of initial conditions $h_{\e}\in G^{\sigma}$ and corresponding solutions $u_{\e}$ of the Cauchy problem on domains $\Omega_{R_{\e},\rho_{\e}}$ such that
	\be
		\lim_{\e\to 0} ||u_{\e}||_{L^2(\Omega_{R_{\e},\rho_{\e}})} / ||h_{\e}||_{\sigma,c,K}^{\a} = +\infty .
	\ee
	
	\noindent The time of existence of the solutions $u_{\e}$ is at least of size $\e^{1/2-\sigma/2} $.
\end{theo}

\begin{theo}[Gevrey ill-posedness of the Airy case]
	\label{theorem.airy}
	Under Assumptions {\rm \ref{hypo.1}}, {\rm \ref{hypo.branching}}, {\rm \ref{hypo.degenerate}} and {\rm \ref{hypo.structural}}, the result of Theorem {\rm \ref{theorem.smooth}} holds for any Gevrey index $\sigma \in (0,\sigma_0)$, with
	$$ \sigma_0 = 2/13 . $$
\end{theo}

\noindent Recall that a function $f$ defined on an open set $B$ of $\R^{d}$ is said to belong to the Gevrey space $G^{\sigma}(B)$ if for all compact $K \subset B$, there are constants $C_K >0$ and $c_K>0$ that satisfy
\be
	|\d^{\a} f|_{L^{\infty}(K)} \leq C_K c_K^{|\a|} |\a|!^{1/\sigma} \quad , \quad \forall \a \in\N^{d}.
\ee

\noindent We then define a family of semi-norms on $G^{\sigma}(B)$, for all compact $K\subset B$ and $c>0$ by
\be
	\label{def.norm.gevrey}
	||f||_{\sigma,c,K} = \sup_{\a} |\d^{\a} f|_{L^{\infty}(K)}c^{-|\a|}|\a|!^{-1/\sigma}.
\ee

\begin{remark}
	\label{remark.un}
	The limiting Gevrey index $\sigma_0$ is in both cases due in part to technical limitations. In the proof, in each case remainder terms are proved to be small in the spaces described later. The limiting index $\sigma_0$ is directly influenced by this smallness of the remainders. In the smooth case, a null remainder would imply $\sigma_0 = 1/2$, which is the expected limiting Gevrey index in this case. In the Airy case, a smaller remainder would imply a greater index $\sigma_0$, but it is not clear if the limit $1/2$ could be attained. 
	
	Also, as pointed out in Remark {\rm \ref{remark.2}}, one main difference between both cases is the extra weight for the Airy propagator in the ansatz of highly oscillating solutions, as shown in Lemma {\rm \ref{lemma.growth.propa.airy}}. This implies a stronger constraint on the smallness of the remainder terms appearing in the Airy case, as explained in the proof of Proposition {\rm \ref{prop.below.airy}}.
\end{remark}

The proofs are given in Sections \ref{section.ansatz} to \ref{section.existence}, with an appendix devoted to the Airy equation in Section \ref{section.appendix}. We introduce a functional framework that is flexible enough to simultaneously cover the smooth, semi-simple case (Theorem \ref{theorem.smooth}) and the stiff, non-semi-simple case (Theorem \ref{theorem.airy}). We develop in Section \ref{section.ansatz} the ansatz of highly oscillating solutions which reduces the Cauchy problem \eqref{Cauchy} to a fixed point equation. In Sections \ref{section.spaces} we recall properties of the spaces developed in \cite{morisse2016I}, and use them to prove contraction estimates and existence of solutions. Finally, in Section \ref{section.hadamard} we prove that the constructed solutions satisfy a lower bound that leads to the Hadamard instability for Gevrey regularity $\sigma\in(0,\sigma_0)$.

%%%%%%%%%%%%%%%%%%%%%%%%%%%%%%%%%%%%%%%%%%%%%%%%%%%%%%%%%%%%%%%%%%%%%%%%%%%%%%%
%%%%%%%%%%%%%%%%%%%%%%%%%%%%%%%%%%%%%%%%%%%%%%%%%%%%%%%%%%%%%%%%%%%%%%%%%%%%%%%

\section{Highly oscillating solutions and reduction to a fixed point equation}
\label{section.ansatz}

%%%%%%%%%%%%%%%%%%%%%%%%%%%%%%%%%%%%%%%%%%%%%%%%%%%%%%%%%%%%%%%%%%%%%%%%%%%%%%%

\subsection{Highly oscillating solutions}

%%%%%%%%%%%%%%%%%%%%%%%%%%%%%%%%%%%%%%%%%%%%%%%%%%%%%%%%%%%%%%%%%%%%%%%%%%%%%%%

As in Section 3.1 of \cite{morisse2016I}, we first reduce \eqref{Cauchy} to the new Cauchy problem
\be
	\label{Cauchy.bis}
	\d_{t}u = \sum_{j}A_{j}(t,x,u)\d_{x_j}{u} + F(t,x,u)u \quad \text{with } u(0,x) = h(x)
\ee

\noindent where $F$ is analytic in a neighborhood of $(0,0,0) \in \R_{t}\times\R_{x}^{d}\times\R^{N}_{u}$ (see \eqref{etc}), and $h$ small analytic functions satisfying $h_{|x=0}=0$, as perturbations of the trivial datum $h\equiv 0$. 

Next, we adapt the ansatz of highly oscillating solutions of \cite{metivier2005remarks} and \cite{morisse2016I} in order to take into account the different time scaling of the exponential growth. In this view we posit
\be
	\label{ansatz}
	u_{\e}(t,x) = \e^{2/(1+\eta)} \u\left(\e^{-1/(1+\eta)}\,t,x,x\cdot\xi_0/\e \right) 
\ee

\noindent where 
\begin{itemize}
	\item The small parameter $\e>0$ corresponds to high frequencies.
	\item The function $\u(s,x,\theta)$ is $2\pi$-periodic in $\theta$.
	\item The scaling term $\e^{2/(1+\eta)}$ insures the smallness of the nonlinear terms.
\end{itemize}

\noindent We introduce for any analytical function $H(t,x,u)$ the compact notation
\be
	\label{def.mathbf.H}
	\H(s,x,\u) = H\left(\e^{1/(1+\eta)}s,x,\e^{2/(1+\eta)}\u\right) .
\ee

\noindent For $u_{\e}(t,x)$ to be solution of \eqref{Cauchy.bis} it is then sufficient that $\u(s,x,\theta)$ solves the following equation
\be
	\label{ds.u.passager}
	\d_{s} \u = \e^{-\eta/(1+\eta)}{\mathbf A} \,\d_{\theta} \u + \e^{1/(1+\eta)}\left( \sum_j {\mathbf A}_j \d_{x_j} \u + {\mathbf F} \, \u\right) 
\ee

\noindent where we use the notation \eqref{def.mathbf.H} for $\mathbf{A}$ and $\mathbf{F}$, and $A$ is defined by \eqref{def.A}.

%%%%%%%%%%%%%%%%%%%%%%%%%%%%%%%%%%%%%%%%%%%%%%%%%%%%%%%%%%%%%%%%%%%%%%%%%%%%%%%

\subsection{Remainder terms}

%%%%%%%%%%%%%%%%%%%%%%%%%%%%%%%%%%%%%%%%%%%%%%%%%%%%%%%%%%%%%%%%%%%%%%%%%%%%%%%

We focus here on the term $\e^{-\eta/(1+\eta)}{\mathbf A} \,\d_{\theta} \u$ of the previous equation. To prove the expected growth of solutions of the initial problem, we decompose the symbol $A(t,x,u)$ in several pieces to highlight the leading term denoted by $\und{A}^{{\rm S}}(t)$ for the smooth case, $\und{A}^{{\rm Ai}}(t)$ for the Airy case, which will lead to the exponential growth. 

First, by analyticity of the $A_j$ and Taylor expansion formula, there is a family of analytical matrices $(A_{u_j})_{j=1,\ldots,N}$ such that locally around $(0,0,0)\in\R_{t}\times\R_{x}^{d}\times\R_{u}^{N}$ there holds
\be
	\label{def.A_u}
	A(t,x,u) = \und{A}(t,x) + \sum_j A_{u_j} u_j .
\ee

\noindent In both smooth and Airy cases, we perform an analytical Taylor expansion on $\und{A}(t,x)$ in order to highlight the principal term that lead to the exponential growth. This is made precise in the two following lemmas.

\begin{lemma}[Expansion formula: smooth case]
	\label{lemma.expansion.smooth}
	Following Proposition {\rm \ref{prop.normalform.smooth}}, we introduce the leading term $\und{A}^{{\rm S}}(t)$, defined up to a change of basis and a trace terme as
	\be
		\label{def.und.A.smooth}
		Q_0^{-1}(t,x) \left( \und{A}^{{\rm S}}(t) - \frac{1}{2} {\rm Tr}\,\und{A}(t,x) \right) Q_0(t,x) =  
		\begin{pmatrix}
			0 & t \\
			-\g_0^2t & 0 
		\end{pmatrix}
	\ee
	
	\noindent and the analytical error term
	\be
		\label{def.rest.smooth}
		Q_0^{-1}\,{\rm R}^{{\rm S}} \, Q_0 =
		\begin{pmatrix}
			0 & 0 \\
			-t(\delta^2 - \delta(0,0)^2) & 0
		\end{pmatrix} .
	\ee
	
	\noindent Then there holds
	\be
		\label{def.R.smooth}
		\und{A}(t,x) = \und{A}^{{\rm S}}(t) + {\rm R}^{{\rm S}}(t,x)
	\ee
	
	\noindent and there are analytical matrices ${\rm R}^{{\rm S}}_t(t,x)$ and ${\rm R}^{{\rm S}}_x(t,x)$ such that
	\be
		\label{expansion.smooth}
		{\rm R}^{{\rm S}}(t,x) = t^2{\rm R}^{{\rm S}}_t(t,x) + tx\cdot{\rm R}^{{\rm S}}_x(t,x)
	\ee
	
	\noindent locally around $(0,0)\in\R_{t}\times\R_{x}^{d}$.
\end{lemma}

\begin{proof}
	First the equality \eqref{normalform.smooth} of Lemma \ref{prop.normalform.smooth} implies that
	$$
		Q_0^{-1}(t,x) \left( \und{A}^{{\rm S}} - \frac{1}{2} {\rm Tr} \,\und{A}(t,x) \right) Q_0(t,x) =  
		\begin{pmatrix}
			0 & t \\
			-t\delta^2 & 0
		\end{pmatrix}
	$$
	
	\noindent hence \eqref{def.R.smooth}. Second, by analyticity of $\delta$ and Taylor expansion formula, there are analytical functions ${\rm r}^{{\rm S}}_t$ and ${\rm r}^{{\rm S}}_{x_j}$ such that
	$$
	%\label{def.rest.taylor}
		\delta^2(t,x) - \delta(0,0)^2 = t\,{\rm r}^{{\rm S}}_t(t,x) + x\cdot{\rm r}^{{\rm S}}_x(t,x) .
	$$

	\noindent We finally introduce the matrices
	$$
	%\label{def.rest.matrix.smooth}
		Q_0^{-1}\,{\rm R}^{{\rm S}}_t \,Q_0= 
		\begin{pmatrix}
			0 & 0 \\
			-{\rm r}^{{\rm S}}_t & 0
		\end{pmatrix}
		\;  \text{and}  \;
		Q_0^{-1}\,{\rm R}^{{\rm S}}_x \,Q_0 = 
		\begin{pmatrix}
			0 & 0 \\
			-{\rm r}^{{\rm S}}_x & 0
		\end{pmatrix}
	$$

	\noindent which leads to \eqref{expansion.smooth} and ends the proof.
\end{proof}

\begin{lemma}[Expansion formula: Airy case]
	\label{lemma.expansion.airy}
	Following Proposition {\rm \ref{prop.normalform.stiff}}, we introduce the leading term $\und{A}^{{\rm S}}(t)$, defined up to a change of basis and a trace term as
	\be
		\label{def.und.A.airy}
		Q_0^{-1}(t,x) \left( \und{A}^{{\rm Ai}}(t) - \frac{1}{2} {\rm Tr}\,\und{A}(t,x) \right) Q_0(t,x) =  
		\begin{pmatrix}
			0 & 1 \\
			-\g_0^2t & 0 
		\end{pmatrix}
	\ee
	
	\noindent and the analytical error term
	\be
		\label{def.rest.airy}
		Q_0^{-1}(t,x) \left( {\rm R}^{{\rm Ai}} - \frac{1}{2} {\rm Tr}\,\und{A}(t,x) \right) Q_0(t,x) =
		\begin{pmatrix}
			0 & 0 \\
			-t(e - e(0,0)) + t_{\star}e & 0
		\end{pmatrix} .
	\ee	
	
	\noindent Then there holds
	\be
		\label{def.R.airy}
		\und{A}(t,x) = \und{A}^{{\rm Ai}}(t) + {\rm R}^{{\rm Ai}}(t,x).
	\ee
	
	\noindent and there are analytical matrices ${\rm R}^{{\rm Ai}}_t$, ${\rm R}^{{\rm Ai}}_x$ and ${\rm R}^{{\rm Ai}}_e$ such that
	\be
		\label{expansion.airy}
		{\rm R}^{{\rm Ai}}(t,x) = t^2{\rm R}^{{\rm Ai}}_t(t,x) + tx\cdot{\rm R}^{{\rm Ai}}_x(t,x) + t_{\star} {\rm R}^{{\rm Ai}}_e(t,x,\xi)
	\ee
	
	\noindent locally around $(0,0)$.
\end{lemma}

\begin{proof}
	First equality \eqref{normalform.stiff} of Lemma \ref{prop.normalform.stiff} implies
	$$
		Q_0^{-1}(t,x) \left( \und{A}^{{\rm Ai}} - \frac{1}{2} {\rm Tr}\,\und{A}(t,x) \right) Q_0(t,x) =
		\begin{pmatrix}
			0 & 1 \\
			-(t-t_{\star}(x))e(t,x) & 0
		\end{pmatrix}
	$$
	
	\noindent hence \eqref{def.R.airy}. Second, by analyticity of $e$ and Taylor expansion formula, there are analytical functions $e_t$ and $e_{x_j}$ such that
	$$
		%\label{def.rest.taylor.airy}
		e(t,x) - e(0,0) = te_t +x\cdot e_x 
	$$

	\noindent locally around $(0,0)$. Introducing the matrices
	$$
		%\label{def.rest.matrix.airy}
		Q_0^{-1}\,{\rm R}^{{\rm Ai}}_t \,Q_0 = 
		\begin{pmatrix}
			0 & 0 \\
			-e_t & 0
		\end{pmatrix}
		\quad , \quad
		Q_0^{-1}\,{\rm R}^{{\rm Ai}}_x \,Q_0 = 
		\begin{pmatrix}
			0 & 0 \\
			-e_x & 0
		\end{pmatrix}
		\quad \text{and} \quad 
		Q_0^{-1}\,{\rm R}^{{\rm Ai}}_e \,Q_0 = 
		\begin{pmatrix}
			0 & 0 \\
			e & 0
		\end{pmatrix}
	$$

	\noindent leads to \eqref{expansion.airy} and ends the proof.
\end{proof}

In both Airy and smooth cases, we have then an expansion formula of the form
$$
	A(t,x,u) = \und{A}^{\eta}(t) + {\rm R}^{\eta}(t,x) + A_{u}\cdot u
$$

\noindent where $\eta$ corresponds to the parameter introduced in the ansatz \eqref{ansatz}, adapted to each specific case. This parameter will be be precised in Lemma \ref{lemma.growth.propa.smooth} in the smooth case, and in Lemma \ref{lemma.growth.propa.airy} in the Airy case. The remainder term ${\rm R}^{\eta}$ is ${\rm R}^{{\rm S}}$ defined by \eqref{def.rest.smooth} in the smooth case, and is ${\rm R}^{{\rm Ai}}$ defined by \eqref{def.rest.airy} in the Airy case.

We rewrite now equation \eqref{ds.u.passager} as
\be
	\label{equation.u.compact}
	\d_{s} \u - \e^{-\eta/(1+\eta)}\und{A}^{\eta}(\e^{1/(1+\eta)} s) \d_{\theta} \u = {\mathbf G}^{\eta}(s,x,\u)
\ee

\noindent where we define the source term
\begin{eqnarray}
	\label{def.G.eta}
	{\mathbf G}^{\eta}(s,x,\u) & = & \e^{-\eta/(1+\eta)}\left(\RR^{\eta} + \e^{2/(1+\eta)} \A_u \cdot\u\right)\d_{\theta} \u \\
		& & + \e^{1/(1+\eta)}\left(\A(s,x,\u)\cdot \d_{x} \u + \F(s,x,\u) \u\right) \nonumber 
\end{eqnarray}

\noindent using notation \eqref{def.mathbf.H}.

\begin{remark}
	\label{remark.4}
	Note that in {\rm \cite{morisse2016I}} there are no remainder terms $\RR^{\eta}$, as we consider the full varying-coefficient operator $\und{A}(\e s , x) \d_{\theta}$ in equation {\rm \eqref{equation.u.compact}}.
\end{remark}

%%%%%%%%%%%%%%%%%%%%%%%%%%%%%%%%%%%%%%%%%%%%%%%%%%%%%%%%%%%%%%%%%%%%%%%%%%%%%%%

\subsection{Upper bounds for the propagators}
\label{subsection.bounds.propa}

%%%%%%%%%%%%%%%%%%%%%%%%%%%%%%%%%%%%%%%%%%%%%%%%%%%%%%%%%%%%%%%%%%%%%%%%%%%%%%%

To solve the Cauchy problem of equation \eqref{equation.u.compact} with initial datum $h_{\e}$ specified in Section \ref{subsection.free.solution}, we first study the case ${\mathbf G}^{\eta} \equiv 0$, that is
\be
	\label{equation.free}
	\d_{s} \u - \e^{-\eta/(1+\eta)}\und{A}^{\eta}(\e^{1/(1+\eta)} s) \d_{\theta} \u = 0 .
\ee

\noindent Note that this equation is linear, non autonomous and non scalar. For a general $\und{A}^{\eta}(t)$ we define the matrix propagator $U^{\eta}(s',s,\theta)$ as the solution of
$$
	\d_{s} U^{\eta}(s',s,\theta) - \e^{-\eta/(1+\eta)}\und{A}^{\eta}(\e^{1/(1+\eta)} s) \d_{\theta} U^{\eta}(s',s,\theta) = 0 \; , \quad U^{\eta}(s',s',\theta) \equiv {\rm Id} 
$$

\noindent and $U^{\eta}(s',s,\theta)$ is periodic in $\theta$, following the ansatz \eqref{ansatz}. The choice of the time scaling $ s = \e^{-1/(1+\eta)} t $, that is the choice of $\eta$, is such that solutions of \eqref{equation.free} have a typical exponential growth independent of $\e$. Both following Lemmas make the growth of the propagators explicit in both cases.

\begin{lemma}[Growth of the propagator: the smooth case]
	\label{lemma.growth.propa.smooth}
	Under Assumptions {\rm \ref{hypo.branching}}, {\rm \ref{hypo.smooth}} and {\rm \ref{hypo.semi-simple}}, we put $\eta = 1 $. The matrix propagator $U^{{\rm S}}(s',s,\theta)$ defined by
	\be
		\label{equation.propagator.smooth}
		\d_{s} U^{{\rm S}}(s',s,\theta) - \und{A}^{{\rm S}}(s) \d_{\theta} U^{{\rm S}}(s',s,\theta) = 0 \quad , \quad U^{{\rm S}}(s',s',\theta) = {\rm Id}
	\ee

	\noindent satisfies the following growth of its Fourier modes in the $\theta$ variable:
	\be
		\label{growth.propa.smooth}
		|U^{{\rm S}}_n(s',s)| \lesssim \exp\left( \int_{s'}^{s} \g^{\sharp}_{{\rm S}}(\tau) d\tau \, |n|\right) \quad , \quad \forall \, 0 \leq s' \leq s \, , \forall \, n\in\Z 
	\ee
	
	\noindent with
	\be
		\label{def.gammasharp.smooth}
		\g^{\sharp}_{{\rm S}}(\tau) = \g_0 \tau .
	\ee
\end{lemma}

\begin{proof}

First, as $\und{A}^{{\rm S}}$ is given by \eqref{def.und.A.smooth} up to a change of basis $Q_0(t,x)$ and a trace term $\frac{1}{2} {\rm Tr}\und{A}$, we introduce 
\be
	\label{local.def.Vn}
	V_n(s',s,x) = \exp \left( in \int_{s'}^{s} \frac{1}{2} {\rm Tr}\und{A}(\e^{1/(1+\eta)} \tau,x) d\tau \right) Q_0^{-1}\left(\e^{1/(1+\eta)} s , x \right) U_n(s',s)
\ee

\noindent which solves
\be
	\label{local.Vn.smooth}
	\d_{s}V_n(s',s) - in 
		\begin{pmatrix}
			0 & \e^{(1-\eta)/(1+\eta)} s \\
			-\g_0^2 \e^{(1-\eta)/(1+\eta)} s & 0 
		\end{pmatrix}
	V_n(s',s) = - \e^{1/(1+\eta)} Q_0^{-1}\d_{t} Q_0 \,V_n(s',s)
\ee

\noindent with initial condition $V_n(s',s',x) = Q_0^{-1}\left(\e^{1/(1+\eta)} s',x\right)$. We focus then on the autonomous differential system
\be
	\label{local.Vntilde.smooth}
	\d_{s}\widetilde{V}_n(s',s) - in 
		\begin{pmatrix}
			0 & \e^{(1-\eta)/(1+\eta)} s \\
			-\g_0^2 \e^{(1-\eta)/(1+\eta)} s & 0 
		\end{pmatrix}
	\widetilde{V}_n(s',s) = 0
\ee

\noindent which becomes, with the choice $\eta = 1$, the $\e$-free matrix equation
$$
	\d_{s}\widetilde{V}_n(s',s) - in \,s \,
		\begin{pmatrix}
			0 & 1 \\
			-\g_0^2 & 0 
		\end{pmatrix}
	\widetilde{V}_n(s',s) = 0 .
$$

\noindent The complex constant change of basis
$$
	\und{Q} =
	\begin{pmatrix}
		1 & 1 \\
		-i\g_0 & i\g_0
	\end{pmatrix}
$$

\noindent leads us to the exact solution
$$
	\widetilde{V}_n(s',s) = \und{Q}
	\begin{pmatrix}
		\exp(n\g_0 (s^2-s'^2)/2) & 0 \\
		0 & \exp(-n\g_0 (s^2-s'^2)/2)
	\end{pmatrix}
	\und{Q}^{-1}
$$

\noindent which satisfies the upper bound
\be
	\label{local.tildeV.bound}
	\left|\widetilde{V}_n(s',s)\right| \lesssim \exp\left( \int_{s'}^{s} \g^{\sharp}_{{\rm S}}(\tau) d\tau \, |n| \right)
\ee

\noindent with $\g^{\sharp}_{{\rm S}}$ defined in \eqref{def.gammasharp.smooth}. Getting back to \eqref{local.Vn.smooth}, we use Duhamel formula to write
\be
	\label{local.smooth.duhamel}
	V_n(s',s) = \widetilde{V}_n(s',s)Q_0^{-1}(\e^{1/2}s') - \int_{s'}^{s} \e^{1/2} \widetilde{V}_n(\tau,s)\, \left(Q_0^{-1}\d_{t} Q_0 \right) (\e^{1/2}\tau)\,V_n(s',\tau) d\tau .
\ee

\noindent Note that $V_n(s',s)$ depends also on $x$ through $Q_0 = Q_0(t,x)$. Factorizing by the exponential growth $\exp( \int_{s'}^{s} \g^{\sharp}_{{\rm S}}(\tau) d\tau \, |n| )$, we get
\begin{align*}
	& \exp\left( -\int_{s'}^{s} \g^{\sharp}_{{\rm S}} \, |n| \right) \left|V_n(s',s)\right| \\
	& \quad \lesssim \, \exp\left( -\int_{s'}^{s} \g^{\sharp}_{{\rm S}}\, |n| \right)  \left|\widetilde{V}(s',s)\right| \\
	& \qquad + \int_{s'}^{s} \e^{1/2} \exp\left( -\int_{\tau}^{s} \g^{\sharp}_{{\rm S}} \, |n| \right) \left|\widetilde{V}_n(\tau,s)\right|\, \left|Q_0^{-1}\d_{t} Q_0 \right| (\e^{1/2}\tau)\, \exp\left( -\int_{s'}^{\tau} \g^{\sharp}_{{\rm S}} \, |n| \right) |V_n(s',\tau)| d\tau 
\end{align*}

\noindent and the bound holds uniformly in $x$. We introduce then 
$$
	y_{\e}(s',s) = \max_{x \in K_{\e}} \exp\left( -\int_{s'}^{s} \g^{\sharp}_{{\rm S}} \, |n| \right) \left|V_n(s',s,x)\right| \quad \text{with } K_{\e} = B_{R^{-1}}(0) .
$$ 

\noindent Thanks to the upper bound \eqref{local.tildeV.bound} there holds
$$
	y_{\e}(s',s) \lesssim 1 + \int_{s'}^{s} \e^{1/2} \, \left|Q_0^{-1}\d_{t} Q_0 \right| (\e^{1/2}\tau)\,y_{\e}(s',\tau) d\tau
$$

\noindent and we make use of the Gronwall inequality to get 
$$
	y_{\e}(s',s) \lesssim \exp\left( \e^{1/2}c\,(s-s') \right) \quad \text{with} \quad c = \max_{x \in K_{\e}} \left|Q_0^{-1}\d_{t} Q_0 \right|
$$

\noindent As $\e^{1/2}s$ is small in our setting, we get the announced upper bound \eqref{growth.propa.smooth}.  

\end{proof}

\vspace{0.5cm}

\begin{lemma}[Growth of the propagator: the Airy case]
	\label{lemma.growth.propa.airy}
	Under Assumptions {\rm \ref{hypo.branching}} and {\rm \ref{hypo.degenerate}} we put $ \eta = 1/2$. The matrix propagator $U^{{\rm Ai}}(s',s,\theta)$ defined by 
	\be
		\label{equation.propagator.airy}
		\d_{s} U^{{\rm Ai}}(s',s,\theta) - \e^{-1/3}\und{A}^{{\rm Ai}}(\e^{2/3} s) \d_{\theta} U^{{\rm Ai}}(s',s,\theta) = 0 \quad , \quad U^{{\rm Ai}}(s',s',\theta) \equiv {\rm Id}
	\ee

	\noindent  satisfies the following growth of its Fourier modes in the $\theta$ variable:
	\be
		\label{growth.propa.airy}
		|U^{{\rm Ai}}_n(s',s)| \lesssim \e^{-1/3}\exp\left( \int_{s'}^{s} \g^{\sharp}_{{\rm Ai}}(\tau) d\tau \, |n|)\right) \quad , \quad \forall \, 0 \leq s' \leq s <\e^{-2/3} \; , \; \forall \, n\in\Z 
	\ee
	
	\noindent with
	\be
		\label{def.gammasharp.airy}
		\g^{\sharp}_{{\rm Ai}}(\tau) = \g_0 \tau^{1/2} .
	\ee
\end{lemma}

\begin{proof}

We proceed as in the previous proof of Lemma \ref{lemma.growth.propa.smooth}, by defining
$$
	V_n(s',s,x) = \exp \left( in \int_{s'}^{s} \frac{1}{2} {\rm Tr}\und{A}(\e^{1/(1+\eta)} \tau,x) d\tau \right) Q_0^{-1}(\e^{1/(1+\eta)} s , x ) U_n^{{\rm Ai}}(s',s,x)
$$

\noindent and looking at the matrix equation
\be
	\label{local.Vntilde.airy}
	\d_{s}\widetilde{V}_n(s',s) - in 
		\begin{pmatrix}
			0 & \e^{-\eta/(1+\eta)} \\
			-\g_0^2 \e^{(1-\eta)/(1+\eta)} s & 0 
		\end{pmatrix}
	\widetilde{V}_n(s',s) = 0 .
\ee

\noindent As the eigenvalues of 
$$
	\begin{pmatrix}
		0 & \e^{-\eta/(1+\eta)} \\
		-\g_0^2\e^{(1-\eta)/(1+\eta)}s & 0 
	\end{pmatrix}
$$

\noindent are $\pm i\g_0 \e^{(1-2\eta)/(1+\eta)}\sqrt{s}$, the choice $\eta = 1/2$ is natural in this case. 

As in the proof of Lemma \ref{lemma.growth.propa.smooth}, to prove the upper bound \eqref{growth.propa.airy} it suffices to prove the same bound for $\widetilde{V}_n(s',s)$. This is postponed to Section \ref{subsection.prop} of the Appendix. It uses classical bounds on the Airy function.

\end{proof}
%
%\begin{remark}
	%\label{remark}
	%The limitation $s < \e^{-2/3}$ for the previous estimate is not a constraint for the following of the paper. Indeed, as $t = \e^{2/3}s$ is bound to be small, $s < \e^{-2/3}$ is valid in our scaling.
%\end{remark}

%%%%%%%%%%%%%%%%%%%%%%%%%%%%%%%%%%%%%%%%%%%%%%%%%%%%%%%%%%%%%%%%%%%%%%%%%%%%%%%

\subsection{Free solutions}
\label{subsection.free.solution}

%%%%%%%%%%%%%%%%%%%%%%%%%%%%%%%%%%%%%%%%%%%%%%%%%%%%%%%%%%%%%%%%%%%%%%%%%%%%%%%

As in Section 3.5 in \cite{morisse2016I}, we seek for high-oscillating, small and well-polarized initial data of the form
\be
	\label{def.initial.data}
	h^{\eta}_{\e}(x) = \e^{2/(1+\eta)}e^{-M(\e)} {\rm Re}\left(e^{ix\cdot\xi_0/\e} \vec{e}^{\,\eta}_{+} + e^{-ix\cdot\xi_0/\e} \vec{e}^{\,\eta}_{-}\right)
\ee

\noindent which correspond in the ansatz \eqref{ansatz} of high-oscillating solutions to
\be
	\label{def.initial.data.ansatz}
	\h^{\eta}_{\e}(\theta) = e^{-M(\e)} {\rm Re}\left(e^{i\theta} \vec{e}^{\,\eta}_{+} + e^{-i\theta} \vec{e}^{\,\eta}_{-}\right) .
\ee

\noindent Here $\vec{e}^{\,\eta}_{+}$ and $\vec{e}^{\,\eta}_{-}$ are vectors chosen in each case such that $U^{\eta}(0,s)\h^{\eta}_{\e}$ satisfies the maximal growth of $U^{\eta}$, in either smooth or Airy case. The parameter $M(\e)$ is chosen such that both the Gevrey norm and the size of $h^{\eta}_{\e}$ is small. Following Lemma 3.3 of \cite{morisse2016I}, we posit
\be
	\label{size.gevrey}
	M(\e) = \e^{-\delta} , \quad \delta \in(\sigma,1) .
\ee

\noindent Remark 3.4 in \cite{morisse2016I} explains in particular the link between the long time of existence of solutions and the Gevrey weight $e^{-M(\e)}$, hence the constraint $\sigma < \delta$. 

We introduce also 
\be
	\label{free.solution.eta}
	\f^{\eta}(s,\theta) = U^{\eta}(0,s) \h^{\eta}_{\e}(\theta)
\ee

\noindent which we call the free solution of equation \eqref{equation.u.compact}, as it solves the equation with ${\mathbf G}^{\eta} \equiv 0$. In both smooth and Airy cases, we prove that for well-chosen $\vec{e}^{\,\eta}_{+}$ and $\vec{e}^{\,\eta}_{-}$, Lemma 3.2 of \cite{morisse2016I} still holds with
\be
	%\label{typical.growth.smooth}
	\g^{\flat}_{\eta}(\tau) = \g_0 \tau^{\eta}
\ee 

\noindent for $\eta = 1$ (corresponding to the smooth case) and $\eta = 1/2$ (corresponding to the Airy case).

\begin{lemma}[Growth of the free solution: the smooth case]
	\label{lemma.growth.freesolution.smooth}
	We define
	\be
		\label{def.vec.smooth}
		\vec{e}_{+}^{\,{\rm S}} = Q_0(0,x)
		\begin{pmatrix}
			1 \\
			-i\g_0
		\end{pmatrix}
		\quad \text{and} \quad
		\vec{e}_{-}^{\,{\rm S}} = Q_0(0,x)
		\begin{pmatrix}
			1 \\
			i\g_0
		\end{pmatrix} .
	\ee
	
	\noindent Then the free solution $\f^{{\rm S}}$ of the smooth case satisfies
	\be
		\label{growth.free.solution.smooth}
		|\f^{{\rm S}}(s)| \approx e^{-M(\e)} e^{\int_{0}^{s} \g_{{\rm S}}^{\flat}(\tau) d\tau} \quad , \quad \forall s \geq 0 
	\ee
	
	\noindent where $\approx$ means equality up to a constant and with
	\be
		\label{def.gamma.flat.smooth}
		\g_{{\rm S}}^{\flat}(\tau) = \g_0 \tau .
	\ee
\end{lemma}

\begin{proof}

	We follow step by step the proof of Lemma \ref{lemma.growth.propa.smooth}. First, it is explicit that there holds
	\be
		\label{local.Vtilde+1}
		\widetilde{V}_{+1}(0,s) 
		\begin{pmatrix}
			1 \\
			-i\g_0
		\end{pmatrix}
		= e^{\g_0 s^2/2} 
		\begin{pmatrix}
			1 \\
			-i\g_0
		\end{pmatrix}
	\ee
	
	\noindent and also
	\be
		\label{local.Vtilde-1}
		\widetilde{V}_{-1}(0,s) 
		\begin{pmatrix}
			1 \\
			i\g_0
		\end{pmatrix}
		= e^{\g_0 s^2/2} 
		\begin{pmatrix}
			1 \\
			i\g_0
		\end{pmatrix} .
	\ee
	
	\noindent Then, thanks to \eqref{local.smooth.duhamel} there holds
	$$
		V_{+1}(0,s)\vec{e}_{+}^{\,{\rm S}}
		= \widetilde{V}_{+1}(0,s)
		\begin{pmatrix}
			1 \\
			-i\g_0
		\end{pmatrix}
		 - \int_{0}^{s} \e^{1/2} \widetilde{V}_{+1}(\tau,s)\, \left(Q_0^{-1}\d_{t} Q_0 \right) (\e^{1/2}\tau)\,V_n(0,\tau)\vec{e}_{+}^{\,{\rm S}} d\tau .
	$$
	
	\noindent Using the upper bounds of $\widetilde{V}$ and $V$ proved in Lemma \ref{lemma.growth.propa.smooth}, we get the following estimate for the integral term
	\begin{align*}
		& \left| \int_{0}^{s} \e^{1/2} \widetilde{V}_{+1}(\tau,s)\, \left(Q_0^{-1}\d_{t} Q_0 \right) (\e^{1/2}\tau)\,V_n(0,\tau)\vec{e}_{+}^{\,{\rm S}} d\tau \right| \\
		& \lesssim \e^{1/2} \int_{0}^{s} \exp( \g_0(s^2 - \tau^2)/2) \exp(\g_0 \tau^2/2) d\tau \\
		& \lesssim \e^{1/2} s \exp(\g_0 s^2/2) .
	\end{align*}
	
	\noindent By \eqref{local.Vtilde+1} and as $\e^{1/2}s$ is small, we get  
	$$
		\left| V_{+1}(0,s)\vec{e}_{+}^{\,{\rm S}} \right| \approx e^{\g_0 s^2/2}
	$$
	
	\noindent and the same holds for $V_{-1}(0,s)$. We end the proof by using formula \eqref{local.def.Vn}.

\end{proof}

In the Airy case, we make a careful analysis of the prefactor term coming from the crossing of eigenvalues.

\begin{lemma}[Growth of the free solution: Airy case]
	\label{lemma.growth.freesolution.airy}
		We define
	\be
		\label{def.vec.airy}
		\vec{e}_{+}^{\,{\rm Ai}} = Q_0(0,x)
		\begin{pmatrix}
			{\rm Ai_1}(0) \\
			-i\e^{1/3}j{\rm Ai_1}'(0)
		\end{pmatrix}
		\quad \text{and} \quad
		\vec{e}_{-}^{\,{\rm Ai}} = Q_0(0,x)
		\begin{pmatrix}
			{\rm Ai_1}(0) \\
			i\e^{1/3}j{\rm Ai_1}'(0)
		\end{pmatrix}
		\quad
	\ee
	
	\noindent where ${\rm Ai_1}$ is the Airy function defined in Lemma {\rm \ref{lemma.scalar.airy}} of the Appendix, and $j=e^{2i\pi/3}$. Then the free solution $\f^{{\rm Ai}}$ of the Airy case satisfies
	\be
		\label{growth.free.solution.airy}
		|\f^{{\rm Ai}}(s)| \approx s^{-1/4} e^{-M(\e)} e^{\int_{0}^{s} \g_{{\rm Ai}}^{\flat}(\tau) d\tau} \quad , \quad \forall 0 \leq s < \e^{-2/3} 
	\ee
	
	\noindent with
	\be
		\label{def.gamma.flat.airy}
		\g_{{\rm Ai}}^{\flat}(\tau) = \g_0 \tau^{1/2} .
	\ee
\end{lemma}

%\begin{remark}
	%We emphasize the fact that, even if $\h_{\e}^{{\rm Ai}}$ has its larger coefficient of order $\e^{-1/3}$, in the estimate \eqref{growth.free.solution.airy} there does not appear an extra $\e^{-1/3}$. This is due to the non-homogeneity with respect to $\e$ of the propagator $U^{{\rm Ai}}$.
%\end{remark}

\noindent We postpone the proof of this Lemma to Section \ref{subsection.free} of the Appendix. 

\begin{remark}
	\label{remark.flat.sharp}
	Note that, on the contrary of {\rm \cite{morisse2016I}}, in each case there holds $\g_{\eta}^{\sharp} = \g_{\eta}^{\flat}$. This is due to the fact that we do not consider here the full varying-coefficient operator $\und{A}(\e^{1/(1+\eta)} s , x)\d_{\theta }$ but the reduced operator $\und{A}^{\eta}(s)\d_{\theta }$, which is a first-order approximation of operator $\und{A}(\e^{1/(1+\eta)} s , x)\d_{\theta }$.
\end{remark}

%%%%%%%%%%%%%%%%%%%%%%%%%%%%%%%%%%%%%%%%%%%%%%%%%%%%%%%%%%%%%%%
%\subsubsection{Smallness of the free solution and Gevrey index}
%%%%%%%%%%%%%%%%%%%%%%%%%%%%%%%%%%%%%%%%%%%%%%%%%%%%%%%%%%%%%%%

%The size of the Gevrey-$\sigma$ norm of the initial data $h_{\e}^{\eta}$ is linked to the exponent $M(\e)$ as shown by the following Lemma, which is Lemma 3.3 of \cite{morisse2016I}
%
%\begin{lemma}
	%\label{size.gevrey.exp}
	%For any $\sigma\in(0,1)$, $c>0$ and $K$ a compact of $\R^{d}$ there holds
	%\be
		%\label{esti.gevrey.free}
		%||h_{\e}^{\eta}||_{\sigma,c,K} \lesssim \e^{2/(1+\eta)}e^{ - M(\e) + \frac{1}{\sigma}(c\e)^{-\sigma}} .
	%\ee
	%
	%\noindent We emphasize that the constant in the previous inequality does not depend on $K$.
%\end{lemma}
%
%\noindent As we need $h_{\e}^{\eta}$ to be small both in Gevrey-$\sigma$ norm and in amplitude, we posit
%
%
%\begin{remark}
	%With the previous definition \eqref{size.gevrey}, the initial data $h_{\e}^{\eta}$ is exponentially small, both in Gevrey-$\sigma$ norm and in absolute value. This last point is of importance, as we need $h_{\e}^{\eta}$ to be small enough to see the exponential growth of the solution it generates in a sufficiently long time $t_{\e}$ to be defined later. A constraint on this final time will lead to a constraint on the size $e^{-M(\e)}$ of $h_{\e}^{\eta}$, and then to a constraint on the admissible Gevrey regularity as $\sigma < \delta$ by \eqref{size.gevrey}.
%\end{remark}

%%%%%%%%%%%%%%%%%%%%%%%%%%%%%%%%%%%%%%%%%%%%%%%%%%%%%%%%%%%%%%%%%%%%%%%%%%%%%%%

\subsection{Fixed point equation}

%%%%%%%%%%%%%%%%%%%%%%%%%%%%%%%%%%%%%%%%%%%%%%%%%%%%%%%%%%%%%%%%%%%%%%%%%%%%%%%

Using the propagator $U^{\eta}(s',s,\theta)$, the free solution \eqref{free.solution.eta} and the Duhamel formula, we can express now \eqref{equation.u.compact} as the fixed point equation
\be
	\label{fixed.point.equation}
	\u(s,x,\theta) = \f^{\eta}(s,\theta) + \int_{0}^{s} U^{\eta}(s',s,\theta) {\mathbf G}^{\eta}(s',\u(s'x,\theta)) ds'
\ee

\noindent where $ {\mathbf G}^{\eta}(\u)$ is defined by \eqref{def.G.eta}. We denote the integral term
\be
	\label{defi.T}
	T^{\eta}(s,\u) = \int_{0}^{s} U^{\eta}(s',s)  {\mathbf G}^{\eta}(s',\u(s')) ds'
\ee

\noindent which we split into three parts thanks to definition \eqref{def.G.eta} like
\begin{align}
	& \quad T^{\eta}(s,\u) \\
	& = \int_{0}^{s} U^{\eta}(s',s) \left[ \left(\e^{-\eta/(1+\eta)}\RR^{\eta} + \e^{(2-\eta)/(1+\eta)}\A_{u}\cdot\u\right)\d_{\theta} \u + \e^{1/(1+\eta)} \left(\A\cdot \d_{x} \u + \F\right) \right] ds' \nonumber \\
	& = T^{[\eta, \theta]}(s,\u) + T^{[\eta, x]}(s,\u) + T^{[\eta, \u]}(s,\u) \label{defi.T.decoupage}
\end{align}

\noindent where we define
\begin{eqnarray}
	T^{[\eta,\theta]}(s,x,\u) & = & \int_{0}^{s} U^{\eta}(s',s) \,\left[\e^{-\eta/(1+\eta)}\RR^{\eta} + \e^{(2-\eta)/(1+\eta)}\A_{u}\cdot\u\right] \,\d_{\theta} \u(s') ds' \label{def.T_theta.eta}\\
	T^{[\eta, x]}(s,x,\u) & = & \int_{0}^{s} U^{\eta}(s',s)\, \left[\e^{1/(1+\eta)} \A(s',x,\u(s'))\right] \cdot \d_{x} \u(s') ds' \label{def.T_x.eta} \\
	T^{[\eta, \u]}(s,x,\u) & = & \int_{0}^{s} U^{\eta}(s',s) \, \left[\e^{1/(1+\eta)} \F(s',x,\u(s'))\right] \, \u(s') ds' \label{def.T_u.eta}
\end{eqnarray}

%By holomorphy of $A\cdot\xi_0$ in a neighborhood of $0$ and $ A\cdot\xi_0(0,0,0) = A_0$ by definition, Taylor equality implies there are analytic functions $A_{x_j}$, $A_{t}$ and $A_{u}$ such that
%\be
	%\label{decompo.A-A_0}
	%(A\cdot\xi_0 - A_0)(t,x,u) = t A_{t}(t,x,u) + x\cdot A_{x}(t,x,u) + u A_{u}(t,x,u)
%\ee
%
%\noindent where we denote $x\cdot A_{x} = \sum_{j} x_j A_{x_j}$.

%%%%%%%%%%%%%%%%%%%%%%%%%%%%%%%%%%%%%%%%%%%%%%%%%%%%%%%%%%%%%%%%%%%%%%%%%%%%%%%

%\subsection{Some elements of the proof}
%\label{subsection.sketch}

%%%%%%%%%%%%%%%%%%%%%%%%%%%%%%%%%%%%%%%%%%%%%%%%%%%%%%%%%%%%%%%%%%%%%%%%%%%%%%%

We have now reduced the initial question of finding a family of initial data $h_{\e}$ generating a family of appropriately growing analytic solutions $u_{\e}$ to the fixed point equation \eqref{fixed.point.equation} for operator $T^{\eta}$. In order to prove Theorems \ref{theorem.smooth} and \ref{theorem.airy} we refer to the proof of Gevrey instability in the case of initial ellipticity in \cite{morisse2016I}. A sketch of the proof can be found in Section 3.6 therein.

\section{Contraction estimates}
\label{section.spaces}

We make use of spaces $\EE$ constructed in Section 4 in \cite{morisse2016I} and their properties developed in Section 5 therein. The method is robust enough to be used in our context of transitions from hyperbolicity to ellipticity. 

%%%%%%%%%%%%%%%%%%%%%%%%%%%%%%%%%%%%%%%%%%%%%%%%%%%%%%%%%%%%%%%%%%%%%%%%%%%%%%%

\subsection{Functional spaces: definitions}
\label{subsection.majoring}

%%%%%%%%%%%%%%%%%%%%%%%%%%%%%%%%%%%%%%%%%%%%%%%%%%%%%%%%%%%%%%%%%%%%%%%%%%%%%%%

%%%%%%%%%%%%%%%%%%%%%%%%%%%%%%%
%\subsubsection{Majoring series}

%We consider formal series of $m$ variables, with complex coefficients that depend eventually on a parameter $y$ in some open domain $\mathcal{O}$ of $\C^{m'}$. We denote such formal series 
%$$
	%\phi(z,y) = \sum_{k\in\N^{m}} \phi_{k}(y) z^{k} \qquad \phi_{k}(y) \in \C \, , \quad \forall \, k\in\N^{m}\, , \; \forall \, y\in\mathcal{O} 
%$$
%
%\noindent where we introduce formal unknowns $z = (z_1,\ldots,z_m)$.  A formal series $\phi(z,y)$ is really a $y$-dependent
%sequence $\left(\phi_k(y)\right)_k$ indexed by $k\in\N^{m}$ . An important parameter is the size $m$ of the indices $k$. We define now the relation of majoring series between two formal series $\phi(z,y)$ and $\psi(Z,y)$, with $z$ and $Z$ denoting $m$ variables.
%
%\begin{defi}[Majoring series]
	%\label{definition.maj}
	%For $\phi(z,y)$ and $\psi(Z,y)$ formal series of respectively variable $z$ and variable $Z$, and $y$ a parameter in some open domain $\mathcal{O}$ of $\C^{m'}$, with furthermore
	%$$ \psi(Z,y) = \sum_{k\in\N^{m}} \psi_k(y) Z^{k} \qquad \psi_k(y) \geq 0 \quad \forall k\in\N^{m} \,,\; \forall y\in\mathcal{O} $$
	%\noindent we define 
	%\be
		%\label{def.majoring.series}
		%\phi(z,y) \prec_{y} \psi(Z,y) \quad \Longleftrightarrow \quad \forall k\in\N^{m}\, ,\; \forall y\in\mathcal{O} \quad |\phi_k(y)| \leq \psi_k(y)
	%\ee
%\end{defi}
%
%\noindent See Section 4.1 in \cite{morisse2016I} for more properties and details.

We refer to Section $4.1$ of \cite{morisse2016I} for definition and properties of the majoring series relation denoted by $\prec$. We recall the definition of $\Phi$ as the reference series in one variable 
$$
 \Phi(z) = \sum_{k\geq 0 } \frac{c_0}{k^2+1} z^{k} 
$$

\noindent with $c_0>0$ such that $\Phi^2 \prec \Phi$. We recall the notation 
\be
	\label{def.Phi_k.t}
	\Phi_{k}(\rho t) = R^{|k|}\sum_{p\in\N} \frac{c_0}{(|k|+p)^2+1} \binom{|k|+p}{k,p} \rho^p t^p  \quad , \quad \forall k\in\N^{d}
\ee

\noindent of the positive coefficients of $\Phi(RX_1+\cdots+RX_d+\rho t)$, with $R$ and $\rho$ both positive parameters. From now on, we will denote for convenience and with an abuse of notation $\Phi(RX+\rho t) = \Phi(RX_1+\cdots+RX_d+\rho t) $. 

Recall that, for any formal series $\phi(t,x) = \sum_{k\in\N^{d}} \phi_{k}(t) x^{k}$ in the $x$ variable, with $t$ a parameter, the notation $\phi(t,x) \prec_{t} \Phi(R X + \rho t)$ means
$$
	|\phi_{k}(t)| \leq \Phi_{k}(\rho t) \quad , \quad \forall \, k\in\N^{d} \; , \; \forall t < \rho^{-1} .
$$

We consider trigonometric series in one variable $\theta$ with coefficients in the space of formal series in $d$ variables $x$ in the sense of Section $4.1$ of \cite{morisse2016I}, and we denote $F_{2d+1}$ the space of all such trigonometric series:
$$
	F_{2d+1} = \left\{ \v(x,\theta) = \sum_{n\in\mathbb{Z}} \v_n(x) e^{in\theta} \;\Big|\; \v_n(x) = \sum_{k\in\N^{d}} \v_{n,k} x^{k} \right\} .
$$ 

We now define fixed time spaces $\EE_{s}$ as in \cite{morisse2016I}, which we slightly modified as to take into account two differences:
\begin{itemize}
	\item The renormalization in time is $t = \e^{1/(1+\eta)} s$ instead of $t = \e s$ in \cite{morisse2016I}.
	\item The growth of the propagator as described in Lemmas \ref{lemma.growth.propa.smooth} and \ref{lemma.growth.propa.airy} still depends on a nonnegative function $\g_{\eta}^{\sharp}(\tau)$, and spaces $\EE$ are precisely developed on such functions. The difference lies in the final growth of time, as it will be more precise in \eqref{def.und.s.1}.
\end{itemize}

\begin{defi}[Fixed time spaces $\EE_{s}$]
	\label{definition.fixed.time}
	Given $\eta\geq 0$, $M'>0$, $R>0$, $\rho >0$, $\beta \in (0,1)$ and $s\in[0,\left(\e^{1/(1+\eta)}\rho\right)^{-1})$, we denote $\EE_{s} = \EE_{s}(\eta,R,\rho,M', \beta)$ the space of trigonometric series $\v\in F_{2d+1}$ such that for some constant $C>0$ there holds
	\be
		\label{defi.EEs}
		\v_n(x) \prec C\frac{\displaystyle c_1}{\displaystyle n^2+1} \exp\left( {}- \left(M'- \int_{0}^{s} \g_{\eta}(\tau) d\tau \right)\left< n \right>\right) \Phi\left(R X+\e^{1/(1+\eta)}\rho s\right) \quad , \quad \forall \,n \in\Z
	\ee
	
	\noindent where we denote 
	\be
		\label{def.gamma.eta}
		\g_{\eta}(\tau) = \g_{\eta}^{\sharp}(\tau) + \beta . 
	\ee
	
	\noindent We define a norm on $\EE_{s}$ by
	\be
		\label{defi.norm.s}
		\|\v\|_{s} = \inf \left\{ C>0 \,|\, \quad \eqref{defi.EEs} \text{  is satisfied  } \right\}
	\ee
\end{defi}

As in the discussion following the Definition 4.7 of \cite{morisse2016I}, we introduce the growth time $\und{s}_1(\eta)$ defined implicitely as
\be
	\label{def.und.s.1}
	M' = \int_{0}^{\und{s}_1(\eta)} \g_{\eta}(\tau) d\tau 
\ee

\noindent and the final time as
\be
	\label{finaltime}
	\und{s}(\eta) = \min\left\{\und{s}_1(\eta), \left(\e^{1/(1+\eta)}\rho\right)^{-1}\right\} 
\ee

\noindent where $\left(\e^{1/(1+\eta)}\rho\right)^{-1}$ is the regularity time. To simplify the notations, in all the following we will omit the parameters $R$, $\rho$, $M'$ and $\beta$ in $\EE_{s}(\eta,R,\rho,M',\beta)$. 

We consider now trigonometric series
$$
	\u(s,x,\theta) = \sum_{n\in\Z} \u_n(s,x) e^{in\theta}
$$

\noindent with coefficients $\u_n(s,x)$ being formal series in $x$ whose coefficients depend smoothly on $s\in[0,\und{s}(\eta)[$. We denote $F_{2d+2}$ the space of all such trigonometric series:
$$
	F_{2d+2} = \left\{ \u(s,x,\theta) = \sum_{n\in\Z} \u_n(s,x) e^{in\theta} \;\Big|\; \u_n(s,x) = \sum_{k\in\N^{d}} \u_{n,k}(s)x^{k} \; \text{with } \u_{n,k}(s) \,C^{\infty} \text{ in } s \right\}
$$

\begin{defi}[Spaces $\EE$]
	We introduce
	\be
		\label{defi.space}
		\EE = \left\{\u\in F_{2d+2} \,|\, \forall \, 0\leq s <\und{s}(\eta) \, , \quad \u(s)\in\EE_{s} \right\}
	\ee
	\noindent and the corresponding norm
	\be
		\label{defi.norm}
		||| \u ||| = \sup_{0\leq s < \und{s}(\eta)} \|\u(s)\|_s
	\ee
\end{defi}

Note that $\u$ being in $\EE$ is equivalent to
\be
	\label{esti.u}
	\u_n(s,x) \prec_{s} |||\u||| \frac{\displaystyle c_1}{\displaystyle n^2+1} \exp\left(\left(M'- \int_{0}^{s} \g_{\eta}(\tau) d\tau \right)\left< n \right>\right) \Phi\left(R X+\e^{1/(1+\eta)}\rho s\right) 
\ee

\noindent for all $n\in\Z$ and $s\in[0,\und{s})$.

For $\u$ valued in $\mathbb{C}^{N}$, $\u\in\EE$ means simply that each component of $\u$ is in $\EE$, and $||| \u |||$ is then the maximum of the norms of the components. We denote the ball of $\EE$ of radius $a$, centered in $\u\in\EE$ by
\be
	\label{defi.ball}
	B_{\EE}(\u,a) = \left\{ \v\in\EE \;|\; |||\v - \u||| < a \right\}
\ee

%%%%%%%%%%%%%%%%%%%%%%%%%%%%%%%%%%%%%%%%%%%%%%%%%%%%%%%%%%%%%%%%%%%%%%%%%%%%%%%

\subsection{Functional spaces: properties}

%%%%%%%%%%%%%%%%%%%%%%%%%%%%%%%%%%%%%%%%%%%%%%%%%%%%%%%%%%%%%%%%%%%%%%%%%%%%%%%

%%%%%%%%%%%%%%%%%%%%%%%%%%%%%%%%
\subsubsection{Basic properties}

We remind here basic properties of spaces $\EE$. The proofs are the same as in \cite{morisse2016I}, as those properties depend only on the nonnegativity of $\g_{\eta}^{\sharp}$.

\begin{prop}[Properties of $\EE_s(\eta)$ and $\EE(\eta)$.]
	\label{prop.properties}

	For any $\eta \in [0,1]$ and $s\in[0,\und{s}(\eta))$, there holds
	\begin{enumerate}
		\item The space $\EE_{s}(\eta)$ is an algebra, and for any $\v$ and $\w$ in $\EE_{s}(\eta)$ there holds
			\be
				\label{esti.algebra.s}
				||\v \w||_{s} \leq ||\v||_{s} \, ||\w||_{s} .
			\ee
		\item The space $\EE_{s}(\eta)$ is a Banach space.
	\end{enumerate}
	
	As an immediate corollary, there holds
	\begin{enumerate}
		\item The space $\EE(\eta)$ is an algebra, and for any $\v$ and $\w$ in $\EE(\eta)$ there holds
			\be
				\label{esti.algebra}
				|||\v \w||| \leq |||\v||| \, |||\w||| .
			\ee
		\item The space $\EE(\eta)$ is a Banach space.
	\end{enumerate}
	
\end{prop}

The action of analytic function $H(t,x,u)$ on the space $\EE$, described in the Lemma 4.13 in \cite{morisse2016I}, still holds as it relies on properties of the majoring series relation and on the definition of $\Phi$ described in Section 4.1 of \cite{morisse2016I}.

\begin{lemma}
	\label{lemma.holomorphic}
	Let $H(t,x,u)$ be an analytical function on a neighborhood of $(0,0,0)\in\R_{t}\times\R_{x}^{d}\times\R_{u}^{N}$. There are constants $C_H>0$, $R_H>0$ and $\rho_H>0$ which depend only on $H$ and $c_0$, such that for all $R\geq R_H$ and $\rho\geq\rho_H$ and for $\e$ small enough,  
	\be
		\label{esti.holomorphic}
		\forall \,\u\in B_{\EE(\eta,R,\rho)}(0,1): \quad |||{\mathbf H}(\cdot,\cdot,\u)|||\lesssim 1
	\ee
	
	\noindent where $\H$ is defined by \eqref{def.mathbf.H} and $|||\cdot|||$ is defined by \eqref{defi.norm}.

\end{lemma}

In the operators $T^{[\theta]}$, $T^{[x]} $ and $T^{[\u]} $ defined by \eqref{def.T_theta.eta}, \eqref{def.T_x.eta} and \eqref{def.T_u.eta}, there appears $A$, ${\rm R}^{\eta}$, $A_{u}$ and $F$, all of which are analytic functions in the variables $(t,x,u)\in\R\times\R^{d}\times\R^{N}$. In the expansion formulas of both ${\rm R}^{{\rm S}}$ and ${\rm R}^{{\rm Ai}}$ there appear also analytical functions ${\rm R}^{{\rm S}}_t$, ${\rm R}^{{\rm S}}_t$, ${\rm R}^{{\rm Ai}}_t$, ${\rm R}^{{\rm Ai}}_x$ and ${\rm R}^{{\rm Ai}}_e$ as in Lemmas \ref{lemma.expansion.smooth} and \ref{lemma.expansion.airy}. The previous Lemma applies:

\begin{coro}
	\label{coro.norms}
	There are constants $R_0$ and $\rho_0$ such that for all $\eta \geq 0$, $R\geq R_0$, $\rho \geq \rho_0$ and $\e$ small enough:
	\be
		\label{action.A}
		\forall \,\u\in B_{\EE(\eta,R,\rho)}(0,1): \qquad |||\H(\cdot, \cdot, \u) ||| \lesssim 1
	\ee
	\noindent with $H$ either equals to $A$, $\und{A}$, ${\rm R}^{{\rm S}}_{t}$, ${\rm R}^{{\rm S}}_{x}$,  ${\rm R}^{{\rm Ai}}_{t}$, ${\rm R}^{{\rm Ai}}_{x}$, ${\rm R}^{{\rm Ai}}_{t_{\star}}$, $A_u$ or $F$.
\end{coro}

%%%%%%%%%%%%%%%%%%%%%%%%%%%%%%%%%%%%%%%%%%%%%%
\subsubsection{Action of propagators on $\EE$}
\label{subsection.action}

To describe the action of both propagators $U^{{\rm S}}$ and $U^{{\rm Ai}}$, we define here more general smooth matrix operators $U^{\eta}(s',s,\theta)$ for $\eta \geq 0$ that act diagonally on $\u\in\EE$ as
$$
	\left(U^{\eta}(s',s,\theta) \u(s')\right)_n = U_n^{\eta}(s',s) \u_n(s') \quad , \quad \forall n\in\Z, \;  0 \leq s' \leq s
$$

\noindent and satisfied the upper bound for their Fourier modes
\be
	\label{growth.U.eta}
	|U^{\eta}_n(s',s)| \lesssim C(U^{\eta}) \exp\left( \int_{s'}^{s} \g_{\eta}^{\sharp}(\tau) d\tau \right) \quad , \quad \forall n\in\Z,\; 0 \leq s' \leq s.
\ee

\noindent for some $C(U^{\eta}) > 0$ depending eventually on $\e$. In the smooth case, thanks to Lemma \ref{lemma.growth.propa.smooth} we have $\eta = 1$ and
\be
	\label{def.CU.smooth}
	C(U^{{\rm S}}) = 1 .
\ee

\noindent whereas for the Airy case, thanks to Lemma \ref{lemma.growth.propa.airy} we have $\eta = 1/2$ and
\be
	\label{def.CU.Airy}
	C(U^{{\rm Ai}}) = \e^{-1/3}
\ee

\noindent For such matrix operators $U^{\eta}$, the same result as Lemma 4.15 in \cite{morisse2016I} still holds:

\begin{lemma}
	\label{lemma.action.U}
	Given $\eta \geq 0$, $\beta > 0$ and $\u$ in $\EE(\eta,\beta)$ then
	\be
		\label{precise.bound.action.U}
		U^{\eta}_n(s',s)\u_n(s',x) \prec_{s',s} C(U^{\eta})\,C_n^{\eta}(s',s) \, ||\u(s')||_{s'} \frac{\displaystyle c_1}{\displaystyle n^2+1} e^{- \left(M' - \int_{0}^{s} \g_{\eta}(\tau) d\tau \right) \left< n \right> } \Phi\left(R X+\e^{1/(1+\eta)}\rho s\right) 
	\ee
	
	\noindent where $C_n^{\eta}(s',s)$ is defined by
	\be 
		\label{defi.Cn}
		C_n^{\eta}(s',s) = \exp\left( - \beta (s-s') \left<n\right> \right) \leq 1
	\ee
	
	\noindent In particular we have
	\be
	\label{bound.lemma.action.U}
		\|U^{\eta}(s',s)\u(s')\|_s \leq C(U^{\eta})\,\|\u(s')\|_{s'} \qquad \forall \, 0 \leq s' \leq s < \und{s}
	\ee
	
\end{lemma}

\noindent The proof is exactly the same as in \cite{morisse2016I}, as it relies only on the definition \eqref{def.gamma.eta} of $\g_{\eta}$. The positive constant $\beta$ acts as a perturbation of $\g_{\eta}^{\sharp}$ and introduces an error term like $e^{-\beta (s-s') |n|}$ in the growth of the $n$-th Fourier mode of the propagator. This explains why the prefactor term $C_n^{\eta}(s',s)$ is exactly the same as in Lemma 4.15 of \cite{morisse2016I}.

\begin{remark}
	The estimate \eqref{bound.lemma.action.U} is not precise enough to show that $T$ is a contraction in $\EE$. The more precise estimate \eqref{precise.bound.action.U} is very important for the estimate \eqref{esti.reg_dtheta} below.
\end{remark}

%%%%%%%%%%%%%%%%%%%%%%%%%%%%%%%%%%%%%%%%%%
\subsubsection{Norm of the free solutions}

In both smooth ($\eta = 1$) and Airy ($\eta=1/2$) cases, we compute the norm in $\EE$ of the free solution $\f^{\eta}$ defined in Lemmas \ref{lemma.growth.freesolution.smooth} and \ref{lemma.growth.freesolution.airy}. The proof of this result is the same as Lemma 4.17 of \cite{morisse2016I}, using the precise estimates described in Lemmas \ref{lemma.growth.freesolution.smooth} and \ref{lemma.growth.freesolution.airy}.

\begin{lemma}[Norm of the free solution]
	\label{lemma.norm.free.solution}
	For $\eta = 1$ or $\eta=1/2$ and $\beta >0$, the free solution $\f^{\eta}$ satisfies
	\be
		\label{norm.free.solution}
		|||\f^{\eta}||| \lesssim e^{M'-M(\e)}
	\ee
\end{lemma}

\begin{remark}
	Note that, on the contrary of estimate {\rm (4.33)} in Lemma {\rm 4.17} of {\rm \cite{morisse2016I}}, the previous estimate is not $|||\f^{{\rm Ai}}||| \lesssim \e^{-1/3} e^{M'-M(\e)}$ in the Airy case, thanks to the more precise estimate {\rm \eqref{growth.free.solution.airy}}.
\end{remark}

%%%%%%%%%%%%%%%%%%%%%%%%%%%%%%%%%%%%%%%%%%%%%%%%%%%%%%%%%%%%%%%%%%%%%%%%%%%%%%%

\subsection{Estimates of remainder terms}

%%%%%%%%%%%%%%%%%%%%%%%%%%%%%%%%%%%%%%%%%%%%%%%%%%%%%%%%%%%%%%%%%%%%%%%%%%%%%%%

As pointed out in Remark \ref{remark.4}, our analysis differs from \cite{morisse2016I} with the presence of extra remainder terms $\RR^{\eta}$. We compute carefully their norms.

\begin{lemma}[Smooth case]
	\label{lemma.remainder.smooth}
	In the framework of Lemma {\rm \ref{lemma.expansion.smooth}}, the norm of the remainder term $\RR^{{\rm S}}$ satisfies
	$$
		||| \e^{-1/2}\RR^{{\rm S}} ||| \lesssim \e^{1/2} \und{s}^2 + \und{s} R^{-1} .
	$$
\end{lemma}

\begin{proof}
	By expansion formula \eqref{expansion.smooth} we have
	$$
		{\rm R}^{{\rm S}}(t,x) = t^2{\rm R}^{{\rm S}}_t(t,x) + tx\cdot{\rm R}^{{\rm S}}_x(t,x) 
	$$
	
	\noindent and then, as $t = \e^{1/2} s$ in the smooth case and by notation \eqref{def.mathbf.H},
	$$
		\e^{-1/2}\RR^{{\rm S}}(s,x) = \e^{1/2}s^2\RR^{{\rm S}}_t(s,x) + s x\cdot\RR^{{\rm S}}_x(s,x).
	$$
	
	\noindent As the norm $|||\cdot|||$ is defined by a supremum in time, first there holds
	$$
		||| \e^{1/2}s^2\RR^{{\rm S}}_t ||| \lesssim \e^{1/2}\und{s}^{2} ||| \RR^{{\rm S}}_t ||| .
	$$
	
	\noindent To get a precise estimate of the term $x\cdot\RR^{{\rm S}}_x(s,x)$, we first note that the coefficients $\Phi_k(\rho t)$ defined by \eqref{def.Phi_k.t} satisfy
	$$
		1 \leq R^{-1}\Phi_k(\rho t) \quad , \quad \forall k \in \N^{d} \text{ with } |k| = 1
	$$
	
	\noindent so that for all $j=1,\ldots,d$ there holds
	$$
		X_j \prec R^{-1}\Phi(RX + \e^{1/2}\rho s) \quad , \quad \forall 0 \leq s < \und{s}.
	$$
	
	\noindent By inequality \eqref{esti.algebra} in Proposition \ref{prop.properties} we get then
	$$
		||| x\cdot\RR^{{\rm S}}_x(s,x) ||| \lesssim R^{-1} ||| \RR^{{\rm S}} ||| .
	$$
	
	\noindent As $C(U^{{\rm S}}) = 1$ by \eqref{def.CU.smooth}, this ends the proof.
\end{proof}

\begin{lemma}[Airy case]
	\label{lemma.remainder.airy}
	In the framework of Lemma {\rm \ref{lemma.expansion.airy}}, the norm of the remainder term $\RR^{{\rm Ai}}$ satisfies
	$$
		||| \e^{-1/3} \RR^{{\rm Ai}} ||| \lesssim \e \und{s}^2 + \e^{1/3}\und{s} R^{-1} + \e^{-1/3} t_{\star}(R^{-1}) .
	$$
\end{lemma}

\begin{proof}
	The proof is the same as the previous one, with the differences that $\eta= 1/2$ and $C(U^{{\rm Ai}}) = \e^{-1/3}$.
\end{proof}

%%%%%%%%%%%%%%%%%%%%%%%%%%%%%%%%%%%%%%%%%%%%%%%%%%%%%%%%%%%%%%%%%%%%%%%%%%%%%%%

\subsection{Contraction estimates}

%%%%%%%%%%%%%%%%%%%%%%%%%%%%%%%%%%%%%%%%%%%%%%%%%%%%%%%%%%%%%%%%%%%%%%%%%%%%%%%

%%%%%%%%%%%%%%%%%%%%%%%%%%%%%%%%%%%%%%
\subsubsection{Regularization results}

A crucial observation is that derivation operators $\d_{\theta}$ and $\d_{x_j}$ are not bounded operators in spaces $\EE$, as explained in Section 5.1 in \cite{morisse2016I}. The main results in our previous paper are the description of the regularization effect of integration in time of derivation operators. These results are precised in Sections 5.2 through 5.4 in \cite{morisse2016I}.

Those results still hold in our setting. We omit the proof of the following Lemmas, but give some indications on how to adapt the proofs of \cite{morisse2016I} of the similar results.

\begin{prop}[Regularization of $\d_{\theta}$]
	\label{reg_dtheta}
	For operator $T^{[\eta,\theta]}$ defined by \eqref{def.T_theta.eta}, for any $\u\in B_{\EE(\eta,R,\rho)}(0,1)$ and for $\beta >0$, there holds
	\be
		\label{esti.reg_dtheta}
		|||T^{[\eta,\theta]}(\u)||| \lesssim  C(U^{\eta}) \beta^{-1} \left( \e^{-\eta/(1+\eta)} |||\RR^{\eta}||| + \e^{(2-\eta)/(1+\eta) } |||\u||| \right)\, |||\u|||
	\ee
\end{prop}

The proof in the same as Proposition 5.4 in \cite{morisse2016I}, as it is based on Lemma \ref{lemma.action.U} and the expression of prefactor $C_n^{\eta}$ defined in \eqref{defi.Cn} which is the same as prefactor (4.31) in Lemma 4.15 in \cite{morisse2016I}.

As remainder terms $\RR^{\eta}$ have different norms in spaces $\EE$, given by Lemmas \ref{lemma.remainder.smooth} and \ref{lemma.remainder.airy} , we give more precisely the following two results:

\begin{coro}[Smooth case]
	\label{reg_dtheta.smooth}
	In the smooth case, thanks to Lemma {\rm \ref{lemma.remainder.smooth}}, there holds
	\be
		|||T^{[{\rm S},\theta]}(\u)||| \lesssim \beta^{-1} \left(\e^{1/2} \und{s}^2 + \und{s} R^{-1} + \e^{1/2} |||\u||| \right) |||\u|||.
	\ee
\end{coro}

\begin{coro}[Airy case]
	\label{reg_dtheta.airy}
	In the smooth case, thanks to Lemma {\rm \ref{lemma.remainder.airy}}, there holds
	\be
		|||T^{[{\rm Ai},\theta]}(\u)||| \lesssim \e^{-1/3}\beta^{-1}\left(\e \und{s}^2 + \e^{1/3}\und{s} R^{-1} + \e^{-1/3} t_{\star}(R^{-1}) + \e ||| \u ||| \right) |||\u||| .
	\ee
\end{coro}

About the regularization of derivation operators $\d_{x_j}$, the proof relies again on the simple computation given in Section 1.2.2 of \cite{morisse2016I}. The difference in the time renormalization in $\Phi(Rx + \e^{1/(1+\eta)} \rho s)$, instead of $\Phi(RX + \e\rho s)$ in \cite{morisse2016I}, is a minor one for the proof.

\begin{prop}[Regularization of $\d_{x_j}$]
	\label{reg_dx}
	For operator $T^{[\eta,x]}$ defined by \eqref{def.T_x.eta} and any $\u\in B_{\EE(\eta,R,\rho)}(0,1)$, there holds
	\be
		\label{esti.reg_dx}
		|||T^{[\eta,x]}(\u)||| \lesssim C(U^{\eta})\,R\rho^{-1}\,|||\u||| .
	\ee
\end{prop}

As $\EE$ is an algebra the operator $T^{[\eta,\u]}$ is directly bounded, with no need of a regularization by time result, on the contrary of operators $T^{[\eta,\theta]}$ and $T^{[\eta,x]}$. The following proposition gives us precisely

\begin{prop}[Nonlinear term]
	\label{reg_u}
	For the operator $T^{[\eta,\u]}$ defined by \eqref{def.T_u.eta}, for any $\u\in B_{\EE}(0,1)$ and $\beta >0$ there holds
	\be
		\label{esti.reg_u}
		|||T^{[\eta,\u]}(\u)||| \lesssim C(U^{\eta})\,\beta^{-1} \e^{1/(1+\eta)} |||\u|||^2 .
	\ee
\end{prop}

%%%%%%%%%%%%%%%%%%%%%%%%%%%%%%%%%%%%%
\subsubsection{Contraction estimates}

Thanks to the results of the previous Section, we prove estimates for operator $T^{\eta}$ defined in \eqref{defi.T}, as in Section 5.5 of \cite{morisse2016I}. We omit once again the proof of this result, as it is the same as Proposition 5.9 in \cite{morisse2016I}.

\begin{prop}[Contraction estimates in $\EE$]
	\label{prop.estimates}
	There are $R_0$, $\rho_0>0$ such that for all $\beta >0$, $R\geq R_0$, $\rho>\rho_0$ and $\e \in (0,1)$, we get the following estimates for all $\u$ and $\v$ in $B_{\EE}(0,1)$:
	\begin{align}
		\label{esti.T}
		& |||T^{\eta}(\u)||| \lesssim C(U^{\eta})\,\left(\beta^{-1} \left(\e^{-\eta/(1+\eta)} |||\RR^{\eta}||| + \e^{1/(1+\eta)} |||\u||| \right) + R\rho^{-1}\right) |||\u||| , \\
		\label{esti.TT}
		& |||T^{\eta}(\u) - T^{\eta}(\v)||| \lesssim C(U^{\eta})\,\left(\beta^{-1} \left(\e^{-\eta/(1+\eta)} |||\RR^{\eta}||| + \e^{1/(1+\eta)} |||\u||| \right) + R\rho^{-1}\right)|||\u - \v|||
	\end{align}
\end{prop}

For convenience we introduce
\be
	\label{def.K.epsilon}
	K^{\eta}(\e) = C(U^{\eta})\,\left(\beta^{-1} \left(\e^{-\eta/(1+\eta)} |||\RR^{\eta}||| + \e^{1/(1+\eta)} |||\f^{\eta}||| \right) + R\rho^{-1}\right) .
\ee

%%%%%%%%%%%%%%%%%%%%%%%%%%%%%%%%%%%%%%%%%%%%%%%%%%%%%%%%%%%%%%%%%%%%%%%%%%%%%%%
%%%%%%%%%%%%%%%%%%%%%%%%%%%%%%%%%%%%%%%%%%%%%%%%%%%%%%%%%%%%%%%%%%%%%%%%%%%%%%%

\section{Estimates from below and Hadamard instability}
\label{section.existence}

%%%%%%%%%%%%%%%%%%%%%%%%%%%%%%%%%%%%%%%%%%%%%%%%%%%%%%%%%%%%%%%%%%%%%%%%%%%%%%%

\subsection{Existence of solutions}

%%%%%%%%%%%%%%%%%%%%%%%%%%%%%%%%%%%%%%%%%%%%%%%%%%%%%%%%%%%%%%%%%%%%%%%%%%%%%%%

Thanks to the Proposition \ref{prop.estimates}, we can now solve the fixed point equation \eqref{fixed.point.equation} in the ball $B_{\EE(\eta)}\left(0,|||\f^{\eta}|||\right)$:

\begin{coro}[Contraction and fixed point in $\EE$]
	\label{coro.fixedpoint} 
	Let $\eta >0$ be fixed. Let $R(\e) > R_0$, $\rho(\e) > \rho_0$, $\beta(\e) >0$ and $\und{s}(\eta)$ be such that 
	\be
		\label{hypo.coro.eta}
		\lim_{\e\to 0} C(U^{\eta})\,\left(\beta^{-1}\left(\e^{-\eta/(1+\eta)} |||\RR^{\eta}||| + \e^{1/(1+\eta)} |||\f^{\eta}||| \right) + R\rho^{-1}\right) = 0
	\ee
	
	\noindent Let assume that the propagator $U^{\eta}$ satisfy the growth \eqref{growth.U.eta}. Then for $\e$ small enough, the fixed point equation \eqref{fixed.point.equation}, with $\f^{\eta}$ defined by \eqref{free.solution.eta}, has a unique solution $\u$ in $B_{\EE(\eta,R,\rho,\beta)}\left(0,2|||\f^{\eta}|||\right)$. This solution satisfies
	\be
		\label{esti.solution.eta}
		|||\u - \f^{\eta}||| \lesssim K^{\eta}(\e) |||\f^{\eta}||| 	
	\ee
	
	\noindent with $K^{\eta}$ defined in \eqref{def.K.epsilon}.
\end{coro}

The proof of the Corollary is straigthforward using the estimates of Proposition \ref{prop.estimates}, under the condition of smallness for $K^{\eta}(\e)$ given by \eqref{hypo.coro.eta}. This Corollary is in some sense an abstract result, as it deals with abstract propagator $U^{\eta}$ with specific growth \eqref{growth.U.eta}, described at the beginning of Section \ref{subsection.action}. We emphazise that $U^{\eta}$, except for both smooth and Airy cases, have not be proved to exist. Corollary \ref{coro.fixedpoint} gives a result on fixed point equations \eqref{fixed.point.equation}, independently of the initial Cauchy equation.

%%%%%%%%%%%%%%%%%%%%%%%%%%%%%%%%%%%%%%%%%%%%%%%%%%%%%%%%%%%%%%%%%%%%%%%%%%%%%%%

\subsection{Bounds from below}

%%%%%%%%%%%%%%%%%%%%%%%%%%%%%%%%%%%%%%%%%%%%%%%%%%%%%%%%%%%%%%%%%%%%%%%%%%%%%%%

From now on we focus on both smooth and Airy cases, for which we have proved the existence and the growth of the propagators (see Lemmas \ref{lemma.growth.propa.smooth} and \ref{lemma.growth.propa.airy}) and the actual growth of the special free solution (see Lemmas \ref{lemma.growth.freesolution.smooth} and \ref{lemma.growth.freesolution.airy}). 

We follow here Section 6.2 of \cite{morisse2016I}. We aim to prove that, in the smooth case, the solutions have the same growth as $\f^{{\rm S}}$ given in Lemma \ref{lemma.growth.freesolution.smooth}, that is
\be
	\label{growth.solution.smooth}
	|\u(s,x,\theta)| \gtrsim e^{-M(\e)} \exp\left(\int_{0}^{s} \g_{{\rm S}}^{\flat}(\tau) d\tau\right) \quad , \quad \forall\,(s,x,\theta) \in \Omega_{R,\e^{1/(1+\eta)}\rho}\times\mathbb{T} 
\ee

\noindent with $\Omega_{R,\e^{1/(1+\eta)}\rho}$ defined by \eqref{defi.Omega}. In the Airy case, thanks to Lemma \ref{lemma.growth.freesolution.airy}, we aim to prove 
\be
	\label{growth.solution.airy}
	|\u(s,x,\theta)| \gtrsim s^{-1/4}e^{-M(\e)} \exp\left(\int_{0}^{s} \g_{{\rm Ai}}^{\flat}(\tau) d\tau\right) \quad , \quad \forall\,(s,x,\theta) \in \Omega_{R,\e^{2/3}\rho}\times\mathbb{T} .
\ee

As in \cite{morisse2016I}, by some computations we prove the pointwise inequality
\be
	\label{esti.local}
	|(\u-\f^{\eta})(s,x,\theta))| \lesssim K^{\eta}(\e)\, C(U^{{\rm \eta}}) e^{-M(\e)} \exp\left( \int_{0}^{s} \g_{\eta}(\tau) d\tau \right)
\ee

\noindent which is inequality (6.5) of \cite{morisse2016I}, holding for all $(s,x,\theta) \in \Omega_{R, \e^{1/(1+\eta)}\rho}\times\mathbb{T} $. Next, by definition \eqref{def.gamma.eta} of $\g_{\eta}$, inequality \eqref{esti.local} becomes
$$
	|(\u-\f^{\eta})(s,x,\theta))| \lesssim K^{\eta}(\e)\, C(U^{{\rm \eta}}) \exp\left( \und{s}\beta + \int_{0}^{s} \left( \g_{\eta}^{\sharp}(\tau) - \g_{\eta}^{\flat}(\tau) \right) d\tau \right) e^{-M(\e)} \exp\left( \int_{0}^{s} \g_{\eta}^{\flat}(\tau) d\tau \right).
$$

\noindent As $\g_{\eta}^{\flat} = \g_{\eta}^{\sharp}$ (see Remark \ref{remark.flat.sharp}), we finally get
\be
	\label{smallness.u-f}
	|(\u-\f^{\eta})(s,x,\theta))| \lesssim K(\e)\, C(U^{{\rm \eta}}) \,e^{ \und{s}\beta } \, e^{-M(\e)} \exp\left( \int_{0}^{s} \g_{\eta}^{\flat}(\tau) d\tau \right).
\ee

Then in order to get \eqref{growth.solution.smooth} or \eqref{growth.solution.airy} thanks to \eqref{smallness.u-f}, limit \eqref{hypo.coro.eta} is not sufficient as the term $e^{ \und{s}\beta }$ could be large as $\e$ tends to $0$, as explained in Section 6.2 in \cite{morisse2016I}. In both cases, we have then a stronger constraint on parameters $R$, $\rho$, $\beta$ and $M'$:

\begin{itemize}
	\item In the smooth case, thanks to Corollary \ref{reg_dtheta.smooth} and \eqref{growth.solution.smooth}, we need
\be
	\label{first.constraint.smooth}
	\lim_{\e \to 0} \,\left(\beta^{-1}\,\left(\e^{1/2}\und{s}^2 + \und{s}R^{-1} + \e^{1/2} e^{M' - M} \right) + R\rho^{-1}\right)e^{\und{s} \beta} = 0 .
\ee

	\item In the Airy case, thanks to Corollary \ref{reg_dtheta.airy} and \eqref{growth.solution.airy}, we need
\be
	\label{first.constraint.airy}
	\lim_{\e \to 0} \,\e^{-1/3} \und{s}^{1/4}\,\left(\beta^{-1} \, \left(\e\und{s}^2 + \e^{1/3}\und{s}R^{-1} + \e^{-1/3}t_{\star}(R^{-1}) + \e^{2/3}e^{M' - M} \right) + R\rho^{-1}\right)e^{\und{s} \beta } = 0 .
\ee

\end{itemize}

The second constraint on the parameters comes from the competition between the characteristic growth time $\und{s}_1(\eta)$ defined in \eqref{def.und.s.1} and the regularity time $(\e^{1/(1+\eta)}\rho)^{-1}$. To see the growth of the solution, hence the instability, we need it to exist on a sufficiently large time compared to the growth time, that is we need $\und{s} $ to be $\und{s}_1(\eta)$. 

As $M'$ is large in the limit $\e\to 0$, the implicit definition \eqref{def.und.s.1} of $\und{s}_1(\eta)$ and definition \eqref{def.gammasharp.smooth} of $\g_{{\rm S}}^{\sharp}$ and definition \eqref{def.gammasharp.airy} of $\g_{{\rm Ai}}^{\sharp}$ lead to the equivalent
\be
	\und{s}_1(\eta) \approx M'^{1/(1+\eta)}
\ee

\noindent for $\eta = 1$ and $\eta = 1/2$. Hence the following constraints:
\begin{itemize}

	\item In the smooth case,
\be
	\label{second.constraint.smooth}
	\lim_{\e\to0} M'^{1/2} \e^{1/2}\rho =0.
\ee

	\item In the Airy case,
\be
	\label{second.constraint.airy}
	\lim_{\e\to0} M'^{2/3} \e^{2/3}\rho =0 .
\ee

\end{itemize}

%Up to now we have not been focused on one case or the other, as the methods and results are more general. But here the two constraints \eqref{first.constraint} an \eqref{second.constraint} can not be solved in a general way, as a key difference between the smooth case and the Airy case is the term $C(U^{\eta})$ which of order $\e^{-1/3}$ in the Airy case. 

We focus now on both cases separately from now on, even if the way we find suitable $R$, $\rho$, $\beta$ and $M'$ which would satisfy both constraints is very similar in both cases. We sum up all of this in the two following Propositions.

\begin{prop}[Estimate from below: smooth case]
	\label{prop.below.smooth}
	With the limitation of the Gevrey index
	\be
		\label{gevrey.smooth}
		\sigma < \delta < 1/3 ,
	\ee
	\noindent both constraints \eqref{first.constraint.smooth} and \eqref{second.constraint.smooth} are satisfied. Then the fixed point equation \eqref{fixed.point.equation} has a unique solution $\u$ in $\EE$ and 
	\be 
		|\u(s,x,\theta)| \gtrsim e^{-M(\e)}e^{\g_0 \frac{1}{2}s^{2}} \quad , \quad \forall\,(s,x,\theta)\in\Omega_{R,\e^{1/2}\rho}\times\mathbb{T}.
	\ee
	\noindent There holds also 
	\be
		\label{final.und.s.smooth}
		\und{s} \approx \e^{-\delta/2} .
	\ee
\end{prop}

\begin{proof}
	As in \cite{morisse2016I}, we use notation $\ll$ defined in (1.32) to rewrite all the constraints in a more useful way. Constraints \eqref{first.constraint.smooth} and \eqref{second.constraint.smooth} are equivalent to 
	\begin{eqnarray}
		\beta^{-1}\e^{1/2}\e^{-\delta}\,e^{\beta \e^{-\delta/2}} & \ll & 1 \\
		\beta^{-1} \e^{-\delta/2} R^{-1}\,e^{\beta \e^{-\delta/2}} & \ll & 1 \\
		\beta^{-1}\e^{1/2} e^{M' - M}\,e^{\beta \e^{-\delta/2}} & \ll & 1 \\
		R\rho^{-1}\,e^{\beta \e^{-\delta/2}} & \ll & 1 \\
		\e^{-\delta/2} \e^{1/2}\rho & \ll & 1 .
	\end{eqnarray}
	
	\noindent This implies first, as in \cite{morisse2016I}, that $\beta$ as to be of size $\e^{\delta/2}$. We then posit
	$$
		\beta = \e^{\delta/2}.
	$$
	
	\noindent We have then 
	\begin{eqnarray}
		\e^{1/2}\e^{-3\delta/2} & \ll & 1 \label{un.smooth} \\
		\e^{-\delta} R^{-1} & \ll & 1 \label{deux.smooth} \\
		\e^{1/2 - \delta/2} e^{M' - M}& \ll & 1 \label{trois.smooth} \\
		R\rho^{-1}& \ll & 1 \label{quatre.smooth} \\
		\e^{1/2-\delta/2}\rho & \ll & 1 \label{cinq.smooth} .
	\end{eqnarray}
	
	\noindent Asymptotic inequality \eqref{un.smooth} is equivalent to the limitation $\delta <1/3$ on the Gevrey index. Next, as $\delta <1$, asymptotic inequality \eqref{trois.smooth} is satisfied as soon as $ M' = M - |\ln(\e)| $.
	
	Finally, inequalities \eqref{deux.smooth}, \eqref{quatre.smooth} and \eqref{cinq.smooth} are equivalent to 
	\be
		\e^{1/2-\delta/2} \ll \rho^{-1} \ll R^{-1} \ll \e^{\delta} .
	\ee
	
	\noindent This chain of asymptotic inequalities is satisfied as soon as $\e^{1/2 - \delta/2} \ll \e^{\delta}$, which is equivalent again to the limitation $\delta <1/3$ of the Gevrey index. Then the choice $R^{-1} = \e^{1/6 + \delta/2}$ and $\rho^{-1} = \e^{1/3}$ satisfies the constraints.
	
\end{proof}

\begin{remark}
	\label{remark.smooth}
	Note that in the case where ${\rm R}^{{\rm S}} \equiv 0$, both {\rm \eqref{un.smooth}} and {\rm \eqref{deux.smooth}} disappear hence the limitation $	\delta <1 $ in place of $\delta <1/3$. This has to be put in parallel of Remark {\rm 2.11} in {\rm \cite{morisse2016I} }, which describes a way to improve the result of Theorem {\rm 3} therein.
\end{remark}

\begin{prop}[Estimate from below: Airy case]
	\label{prop.below.airy}
	With the limitation of the Gevrey index
	\be
		\label{gevrey.airy}
		\sigma < \delta < 2/13
	\ee
	\noindent both constraints \eqref{first.constraint.airy} and \eqref{second.constraint.airy} are satisfied. Then the fixed point equation \eqref{fixed.point.equation} has a unique solution $\u$ in $\EE$ and 
	\be 
		|\u(s,x,\theta)| \gtrsim e^{-M(\e)}e^{\g_0 \frac{2}{3}s^{3/2}} \quad , \quad \forall\,(s,x,\theta)\in\Omega_{R,\e^{2/3}\rho}\times\mathbb{T}.
	\ee
	\noindent There holds also
	\be
		\label{final.und.s.airy}
		\und{s} \approx \e^{-2\delta/3} .
	\ee
\end{prop}

%\begin{remark}
	%\label{remark.trois}
	%The exponents of $R$ and $\rho$ in \eqref{good.parameters.airy} are quite esoteric. One main observation is that, in the case $t_{\star}(x) \approx x^2$, the maximal time for which the system is hyperbolic is of order $ \left(R^{-1}\right)^2 \sim \e^{4/9 - \delta/9} $ and larger than the observation time of the instability, that is $ \e^{2/3}\und{s} \sim \e^{2/3 - 2\delta/3}$, as soon as $\delta <2/5$. This is one key to understand why the non-degenerate Airy case is out of the reach of our method.
%\end{remark}

\begin{proof}
	We follow here the same proof as the one of Proposition \ref{prop.below.smooth}. Two differences appear: the extra weight $\e^{-1/3}$ coming from the specificity of the Airy propagator in Lemma \ref{lemma.growth.propa.airy} and the extra $t_{\star}(R^{-1})$ in the remainder term of Corollary \ref{reg_dtheta.airy}.
	
	Taking both those differences into account, the following constraints hold
	\begin{eqnarray}
		\beta^{-1}\e^{2/3}\e^{-3\delta/2}\,e^{\beta \e^{-2\delta/3}} & \ll & 1 \label{un.airy} \\
		\beta^{-1} \e^{-5\delta/6} R^{-1}\,e^{\beta \e^{-2\delta/3}} & \ll & 1 \label{deux.airy} \\
		\beta^{-1} \e^{-\delta/6} \e^{-2/3}t_{\star}(R^{-1})\,e^{\beta \e^{-2\delta/3}} & \ll & 1 \label{trois.airy} \\
		\beta^{-1} \e^{-\delta/6} \e^{1/3}e^{M'-M}\,e^{\beta \e^{-2\delta/3}} & \ll & 1 \label{quatre.airy} \\
		\e^{-1/3} \e^{-\delta/6} R\rho^{-1}\,e^{\beta \e^{-2\delta/3}} & \ll & 1 \label{cinq.airy} \\
		\e^{-2\delta/3} \e^{2/3}\rho & \ll & 1 \label{six.airy}
	\end{eqnarray}
	
	\noindent where there holds $	\und{s} \approx \e^{-2\delta/3}$. Again, those asymptotic inequalities imply that $\beta = \e^{2\delta/3}$ and inequality \eqref{un.airy} is replaced by
	$$ 
		\e^{2/3-13\delta/6} \ll 1 
	$$
	
	\noindent which gives the limitation $\delta <4/13$ on the Gevrey index.
	
	To find now $R$ and $\rho$, we first use \eqref{deux.airy} and \eqref{cinq.airy} to get
	\be
		\label{sept.airy}
		\e^{2/3-2\delta/3} \ll \rho^{-1} \ll \e^{1/3+\delta/6}R^{-1} .
	\ee
	
	\noindent For an asymptotic upper bound for $R^{-1}$, we have in this case two possibilities, thanks to \eqref{deux.airy} and \eqref{trois.airy}. If we assume here that $t_{\star}$ is of order $k \geq 2$ in $x$, these two inequalities are equivalent to
	\be
		\label{local.premier}
		R^{-1} \ll \e^{3\delta/2}
	\ee
	
	\noindent and
	\be
		\label{local.second}
		 R^{-1} \ll \e^{\frac{1}{k}(2/3+5\delta/6)} .
	\ee
	
	\noindent The question is which one of \eqref{local.premier} or \eqref{local.second} is a stronger constraint on $R$ and $\rho$. By simple computations, we prove
	$$
		3\delta/2 < \frac{1}{k}(2/3+5\delta/6) \quad \Longleftrightarrow \quad \delta < \frac{1}{9k/4 - 5/4}.
	$$
	
	\noindent We are then reduced to study two cases:
	\begin{itemize}
	
		\item If $\delta < \frac{1}{9k/4 - 5/4}$, then $R^{-1} \ll \e^{\frac{1}{k}(2/3+5\delta/6)} \ll \e^{3\delta/2}$. With \eqref{sept.airy}, we get the constraint $ \e^{2/3-2\delta/3} \ll \e^{1/3 + \frac{1}{k}(2/3+5\delta/6)} $, which is equivalent to $ \delta < \frac{1/2 - 1/k}{1 + 1/k} $. We note in particular that the non degenerate Airy case $k=2$ is out of reach with our method. In the degenerate case $k=4$, we have the limitation $\delta < \frac{1/2 - 1/k}{1 + 1/k} = 4/21$, compatible with $ \delta < \frac{1}{9k/4 - 5/4} = 4/31 $.
		
		\item If $\delta > \frac{1}{9k/4 - 5/4}$, then $R^{-1} \ll \e^{3\delta/2} \ll \e^{\frac{1}{k}(2/3+5\delta/6)} $. With \eqref{sept.airy}, we get the constraint $ \e^{2/3-2\delta/3} \ll \e^{1/3 + 3\delta/2} $, which is equivalent to $ \delta < 2/13 $. It is incompatible with $\delta > \frac{1}{9k/4 - 5/4}$ when $k=2$.
		
	\end{itemize}
	
	In each of the previous cases, the case $k=2$ leads to a contradiction, hence proving that the non-degenerate Airy transition is out of reach of our method. In the degenerate case $k=4$, the previous analysis shows that the limiting Gevrey index is $2/13$.
	
\end{proof}

\begin{remark}
	On the contrary of the smooth case and Remark {\rm \ref{remark.smooth}}, the limitation $\delta <2/13$ still holds when ${\rm R}^{{\rm Ai}} \equiv 0$. This can be explained as $t_{\star}$ represents the transition time from hyperbolicity to ellipticity, and the domain of hyperbolicity is too large to be considered as an elliptic region. 
\end{remark}

%%%%%%%%%%%%%%%%%%%%%%%%%%%%%%%%%%%%%%%%%%%%%%%%%%%%%%%%%%%%%%%%%%%%%%%%%%%%%%%

\subsection{Conclusion: Hadamard instability in Gevrey spaces}
\label{section.hadamard}

%%%%%%%%%%%%%%%%%%%%%%%%%%%%%%%%%%%%%%%%%%%%%%%%%%%%%%%%%%%%%%%%%%%%%%%%%%%%%%%

To close the proofs of Theorem \ref{theorem.smooth} and Theorem \ref{theorem.airy} we have now to get an estimate of the ratio
$$ \frac{\dsp ||u_{\e}||_{L^2(\Omega_{R,\rho})}}{\dsp ||h_{\e}^{\eta}||^{\a}_{\sigma,c,K}} .$$

\noindent The previous Sections show the existence, in either the smooth or the Airy case, of a family of solutions $\u$ starting from $\f^{\eta}$ of the fixed point equation \eqref{fixed.point.equation}. Thanks to the ansatz \eqref{ansatz} which we recall here
$$
	u_{\e}(t,x) = \e^{2/(1+\eta)}\u(\e^{-1/(1+\eta)}\,t,x,x\cdot\xi_0/\e)
$$

\noindent we have then a family of solutions $u_{\e}$ existing in domains $\Omega_{R,\rho}$, for some well-chosen parameters described in the proof of Proposition {\rm \ref{prop.below.smooth}} or {\rm \ref{prop.below.airy}}. In both cases we can verify that domains $\Omega_{R,\rho}$ contain the cube of size $\e$
$$ 
	C_{\e} = \{ (t,x) \,|\, \und{t} - \e < t <\und{t} ,\quad |x| < \e \} 
$$

\noindent where we denote simply $ \und{t} = \e^{1/(1+\eta)}  \und{s}$. The conclusion of the proof of Theorems \ref{theorem.smooth} and \ref{theorem.airy} is the same as in Section 7 in \cite{morisse2016I}.

\section{Appendix: on the Airy equation}
\label{section.appendix}
%\addcontentsline{toc}{section}{Appendix: on the Airy equation}

%%%%%%%%%%%%%%%%%%%%%%%%%%%%%%%%%%%%%%%%%%%%%%%%%%%%%%%%%%%%%%%%%%%%%%%%%%%%%%%%

The purpose of this Appendix is to bring some crucial elements on the Airy equation, and to complete the proofs of Lemma \ref{lemma.growth.propa.airy} and \ref{lemma.growth.freesolution.airy}. We recall here equation \eqref{local.Vntilde.airy}, with $\eta = 1/2$, that appears on Lemma \ref{lemma.growth.propa.airy}:
\be
	\label{eq.Un.airy}
	\d_{s} \widetilde{V}_n(s',s) - in
		\left(
	\begin{array}{cc}
		0 & \e^{-1/3} \\
		-\e^{1/3}\g_0^{\,2}s & 0 
	\end{array}
	\right)
	\widetilde{V}_n(s',s) = 0 \qquad \widetilde{V}_n(s',s') = \Id .
\ee

\noindent The aim here is to get upper bounds for the matrix flow $\widetilde{V}_n(s',s)$ for all $0 \leq s' \leq s$ and all $n\in\mathbb{Z}$, and hence to complete the proof of Lemma \ref{lemma.growth.propa.airy}. For simplicity we denote 
\be
	\widetilde{V}_n(s',s) =
	\left(
		\begin{array}{cc}
			\widetilde{V}_{n,1,1} & \widetilde{V}_{n,1,2} \\
			\widetilde{V}_{n,2,1} & \widetilde{V}_{n,2,2}
		\end{array}
	\right)
	\qquad \widetilde{V}_{n,p,q}(s',s') = \delta(p,q) .
\ee 

%%%%%%%%%%%%%%%%%%%%%%%%%%%%%%%%%%%%%%%%%%%%%%%%%%%%%%%%%%%%%%%%%%%%%%%%%%%%%%%

\subsection{Reduction to the scalar Airy equation and resolution}

%%%%%%%%%%%%%%%%%%%%%%%%%%%%%%%%%%%%%%%%%%%%%%%%%%%%%%%%%%%%%%%%%%%%%%%%%%%%%%%

The vector equation \eqref{eq.Un.airy} becomes the system of scalar equations
\be
	\label{airy.system}
	\left\{
		\begin{array}{cclcc}
			\d_{s}\widetilde{V}_{n,1,1} & = & \phantom{{}-{}}in \e^{-1/3} \, \widetilde{V}_{n,2,1} & \quad & \widetilde{V}_{n,1,1}(s',s') = 1 \\
			\d_{s}\widetilde{V}_{n,1,2} & = & \phantom{{}-{}}in \e^{-1/3} \, \widetilde{V}_{n,2,2} & \quad & \widetilde{V}_{n,1,2}(s',s') = 0 \\[3mm]
			\d_{s}\widetilde{V}_{n,2,1} & = & {}- in \e^{1/3}\g_0^{\,2}s \, \widetilde{V}_{n,1,1} & \quad & \widetilde{V}_{n,2,1}(s',s') = 0 \\
			\d_{s}\widetilde{V}_{n,2,2} & = & {}- in \e^{1/3}\g_0^{\,2}s \, \widetilde{V}_{n,1,2} & \quad & \widetilde{V}_{n,2,2}(s',s') = 1
		\end{array}
	\right.
\ee

\noindent Differentiating the first equation and using next the third one, the entry $\widetilde{V}_{n,1,1}$ solves the second order scalar differential equation
\be
	\label{second.order.1.1}
	\d_{s}^2 \widetilde{V}_{n,1,1}(s',s) = (n\g_0)^2 \,s \,\widetilde{V}_{n,1,1}(s',s)
\ee

\noindent with the initial condition for $ \widetilde{V}_{n,1,1} $:
\be
	\label{init.1.1}
	\widetilde{V}_{n,1,1}(s',s') = 1 .
\ee

\noindent The initial condition for $ \d_{s}\widetilde{V}_{n,1,1} $ comes from the first equation of the system \eqref{airy.system} and the initial condition for $\widetilde{V}_{n,2,1}$, as we have
\begin{eqnarray}
	\d_{s}\widetilde{V}_{n,1,1}(s',s') & = & in \e^{-1/3} \, \widetilde{V}_{n,2,1}(s',s') \nonumber\\
		& = & 0 . \label{init.1.1.ds}
\end{eqnarray}

\noindent Note also that we can retrieve $\widetilde{V}_{n,2,1}$ thanks to the first line of \eqref{airy.system}, as
\be
	\label{retrieve.2.1}
	\widetilde{V}_{n,2,1}(s',s) = \frac{\e^{1/3}}{in} \, \d_{s}\widetilde{V}_{n,1,1}(s',s)
\ee

\noindent Doing the same for the second and fourth equations, we obtain the same second order scalar differential equation for $\widetilde{V}_{n,1,2}$
\be
	\label{second.order.1.2}
	\d_{s}^2 \widetilde{V}_{n,1,2}(s',s) = (n\g_0)^2 \,s \,\widetilde{V}_{n,1,2}(s',s)
\ee

\noindent with initial conditions
\be
	\label{init.1.2}
	\widetilde{V}_{n,1,2}(s',s') = 0
\ee

\noindent and 
\be
	\label{init.1.2.ds}
	\d_{s}\widetilde{V}_{n,1,2}(s',s') = in\e^{-1/3} .
\ee

\noindent We have also the relation
\be
	\label{retrieve.2.2}
	\widetilde{V}_{n,2,2}(s',s) = \frac{\e^{1/3}}{in} \, \d_{s}\widetilde{V}_{n,1,2}(s',s) .
\ee
	
\noindent Equations \eqref{second.order.1.1} and \eqref{second.order.1.2} are exactly the $\e$-independent scalar Airy equation
\be
	\label{def.Airy.n}
	y_n''(s) = (|n|\g_0)^2 \,s\,y_n(s)
\ee

\noindent which is a second-order scalar differential equation. The solutions of \eqref{def.Airy.n} are given by the following

\begin{lemma}[Scalar Airy equation]
	\label{lemma.scalar.airy}
	For $n\in\mathbb{Z}^{*}$, let ${\rm Ai_n}(z)$ be 
	\be
		\label{def.ain2}
		{\rm Ai_n}(z) = (2\pi)^{-1}\int_{{\rm Im}(\zeta) = a} \exp\left((|n|\g_0)(i\zeta^3/3 + i\zeta z)\right)d\zeta
	\ee
	\noindent for $a>0$. Then for all $n\in\Z$, ${\rm Ai_n}$ is a holomorphic function in $\C$ independent of $a$, and the couple $({\rm Ai_n}(\cdot),{\rm Ai_n}(j\cdot))$ is a basis of solutions of \eqref{def.Airy.n}, with $j=e^{2i\pi/3}$.
\end{lemma}

To prove this Lemma, it suffices to adapt the proof following Definition 7.6.8 in \cite{hormander1983analysis}.

As $({\rm Ai_n}(\cdot),{\rm Ai_n}(j\cdot))$ is a basis of solutions of equation \eqref{def.Airy.n}, and as both entries $\widetilde{V}_{n,1,1}(s',s)$ and $\widetilde{V}_{n,1,2}(s',s)$ solve equations \eqref{second.order.1.1} and \eqref{second.order.1.2}, there are $(\a_1(s'),\beta_1(s'))\in\R^2$ and $(\a_2(s'),\beta_2(s'))\in\R^2$ such that
\begin{eqnarray}
	\widetilde{V}_{n,1,1}(s',s) & = & \a_1(s') {\rm Ai_n}(s) + \beta_1(s'){\rm Ai_n}(js) \label{1.1}\\
	\widetilde{V}_{n,1,2}(s',s) & = & \a_2(s') {\rm Ai_n}(s) + \beta_2(s'){\rm Ai_n}(js) \label{1.2}
\end{eqnarray}

\noindent By the relations \eqref{retrieve.2.1} and \eqref{retrieve.2.2}, there holds also
\begin{eqnarray}
	\widetilde{V}_{n,2,1}(s',s) & = & \frac{\dsp \e^{1/3}}{\dsp in}\left(\a_1(s') {\rm Ai_n}'(s) + j\beta_1(s'){\rm Ai_n}'(js)\right) \label{2.1}\\
	\widetilde{V}_{n,2,2}(s',s) & = & \frac{\dsp \e^{1/3}}{\dsp in}\left(\a_2(s') {\rm Ai_n}'(s) + j\beta_2(s'){\rm Ai_n}'(js)\right) . \label{2.2}
\end{eqnarray}

\noindent This is equivalent to say that both vectors
\be
	\label{basis.airy}
	\begin{pmatrix}
		{\rm Ai_n}(s) \\
		-in^{-1}\e^{1/3} {\rm Ai_n}'(s)
	\end{pmatrix}
	\quad\text{and}\quad
	\begin{pmatrix}
		{\rm Ai_n}(js) \\
		-in^{-1}\e^{1/3}j {\rm Ai_n}'(js)
	\end{pmatrix}
\ee

\noindent forms a basis of solutions of the system \eqref{airy.system}. The functions $(\a_k(s'),\beta_k(s'))$ are determined by the initial conditions \eqref{init.1.1} and \eqref{init.1.1.ds} for $k=1$ and \eqref{init.1.2} and \eqref{init.1.2.ds} for $k=2$. We obtain the matrix representation of the $(\a_k(s'),\beta_k(s'))$:
\be
	\label{matrix.alpha.beta}
	\left(
		\begin{array}{cc}
			\a_1(s') & \a_2(s') \\[3mm]
			\beta_1(s') & \beta_2 (s')
		\end{array}
	\right)
	=
	\frac{\dsp 1}{\dsp D_n(s')}
	\left(
		\begin{array}{cc}
			\frac{\dsp \e^{1/3}}{\dsp in}j{\rm Ai_n}'(js') & -{\rm Ai_n}(js') \\[3mm]
			-\frac{\dsp \e^{1/3}}{\dsp in}{\rm Ai_n}'(s') & {\rm Ai_n}(s')
		\end{array}
	\right)
\ee

\noindent where $D_n(s')$ is the determinant of the basis \eqref{basis.airy}, that is 
$$ D_n(s') :=\frac{\dsp \e^{1/3}}{\dsp in}{\rm Ai_n}(s') j{\rm Ai_n}'(js') - \frac{\dsp \e^{1/3}}{\dsp in}{\rm Ai_n}(js') {\rm Ai_n}'(s') $$

\noindent which is in fact independent of $s'$:
$$ D_n(s') \equiv D_n(0) = \frac{\dsp \e^{1/3}}{\dsp in}(j-1){\rm Ai_n}(0){\rm Ai_n}'(0) .  $$

\noindent For simplicity we denote 
\be
	C = \left((j-1){\rm Ai_n}(0){\rm Ai_n}'(0)\right)^{-1} .
\ee 

Putting altogether equalities \eqref{1.1} to \eqref{matrix.alpha.beta}, we obtain
\be
	\label{matrix.U.Ai}
	\left\{
		\begin{array}{ccl}
			\widetilde{V}_{n,1,1}(s',s) & = & \quad \quad \,\, C\Big(j{\rm Ai_n}'(js') {\rm Ai_n}(s) - {\rm Ai_n}'(s') {\rm Ai_n}(js)\Big) \\[3mm]
			\widetilde{V}_{n,2,1}(s',s) & = & C\frac{\dsp \e^{1/3}}{\dsp in}\,\Big(j{\rm Ai_n}'(js') {\rm Ai_n}'(s) - j{\rm Ai_n}'(s'){\rm Ai_n}'(js)\Big) \\[3mm]
			\widetilde{V}_{n,1,2}(s',s) & = & C\frac{\dsp in}{\dsp \e^{1/3}}\,\,\,\Big(-{\rm Ai_n}(js'){\rm Ai_n}(s) + {\rm Ai_n}(s'){\rm Ai_n}(js)\Big) \\[3mm]
			\widetilde{V}_{n,2,2}(s',s) & = & \quad\quad-C\Big({\rm Ai_n}(js'){\rm Ai_n}'(s) - j{\rm Ai_n}(s'){\rm Ai_n}'(js)\Big) .
		\end{array}
	\right.
\ee

%%%%%%%%%%%%%%%%%%%%%%%%%%%%%%%%%%%%%%%%%%%%%%%%%%%%%%%%%%%%%%%%%%%%%%%%%%%%%%%

\subsection{Upper bounds for the propagator: proof of Lemma (3.4)}
\label{subsection.prop}

%%%%%%%%%%%%%%%%%%%%%%%%%%%%%%%%%%%%%%%%%%%%%%%%%%%%%%%%%%%%%%%%%%%%%%%%%%%%%%%

In order to prove Lemma \ref{lemma.growth.propa.airy}, we derive asymptotic estimates of ${\rm Ai_n}(s)$ and ${\rm Ai_n}(js)$ when $s$ real and $s\to {}+ \infty$. 

\begin{lemma}[Asymptotic estimates for the Airy function]
	\label{lemma.asymptotics}
	There holds for all $n\in\Z^{*}$ and $s\geq 1$, up to some complex constants:
		\begin{eqnarray}
		\label{asymptotic.ai_n.z}
		{\rm Ai_n}(s)   & \approx & s^{-1/4}|n|^{-1/2}\,\exp(-|n|\g_0 (2/3)s^{3/2}) \\
		\label{asymptotic.ai_n.jz}
		{\rm Ai_n}(js)  & \approx & s^{-1/4}|n|^{-1/2} \, \exp(|n|\g_0 (2/3)s^{3/2})  \\
		\label{asymptotic.ai'_n.z}
		{\rm Ai_n}'(s)  & \approx & s^{1/4}|n|^{1/2} \, \exp(-|n|\g_0 (2/3)s^{3/2}) \\
		\label{asymptotic.ai'_n.jz}
		{\rm Ai_n}'(js) & \approx &  s^{1/4}|n|^{1/2} \, \exp(|n|\g_0 (2/3)s^{3/2})  .
\end{eqnarray}

	\noindent In particular, the Airy function ${\rm Ai_n}$ and its derivative satisfy the upper bounds
	\be
		\label{bound.ai_n}
		e^{|n|\g_0 (2/3)s^{3/2}}\left|{\rm Ai_n}(s)\right| + e^{-|n|\g_0 (2/3)s^{3/2}}\left|{\rm Ai_n}(js)\right| \lesssim |n|^{-1/2}(1+s)^{-1/4} \quad \forall 0 \leq s ,\; \forall n\in\Z^*
	\ee
	
	\noindent and
	\be
		\label{bound.ai_n'}
		e^{|n|\g_0 (2/3)s^{3/2}}\left|{\rm Ai_n}'(s)\right| + e^{-|n|\g_0 (2/3)s^{3/2}}\left|{\rm Ai_n}'(js)\right|  \lesssim |n|^{1/2}s^{1/4} \quad \forall 0 \leq s ,\; \forall n\in\Z^*
	\ee

\end{lemma}

\begin{proof}

For $s \geq 1$, we put $a= is^{1/2}$ into the definition \eqref{def.ain2} to obtain
\begin{eqnarray*}
	{\rm Ai_n}(s) & = & (2\pi)^{-1}\int_{{\rm Im}\zeta = is^{1/2}} \exp\left((|n|\g_0)(i\zeta^3/3 + i\zeta z)\right)d\zeta \\
		& = & (2\pi)^{-1}\int_{\mathbb{R}} \exp\left((|n|\g_0)(i(\xi+ i s^{1/2})^3/3 + i(\xi+ i s^{1/2}) s)\right)d\xi \\
		& = & (2\pi)^{-1} \, e^{-|n|\g_0 (2/3)s^{3/2}} \, \int_{\mathbb{R}} \exp\left((|n|\g_0)(i\xi^3/3 - \xi^2s^{1/2})\right)d\xi .
\end{eqnarray*}	

\noindent By the change of variables $ \xi \mapsto (|n|\g_0s^{1/2})^{-1/2}\xi $ in the integral, there holds
$$
	{\rm Ai_n}(s) = \frac{|n|^{-1/2}s^{-1/4}}{2\pi\g_0^{1/2}}\, e^{-|n|\g_0 (2/3)s^{3/2}} \, \int_{\mathbb{R}} \exp\left(i(|n|\g_0)^{{}-1/2}s^{{}-3/4}\xi^3/3 - \xi^2\right)d\xi.
$$

\noindent As the last integral satisfies the asymptotic development, for $s\to{}+\infty$:
\begin{eqnarray*}
	\int_{\mathbb{R}} \exp\left(i(|n|\g_0)^{{}-1/2}s^{{}-3/4}\xi^3/3 - \xi^2\right)d\xi = \sqrt{2\pi} + O\left(|n|^{-1/2}s^{-3/4}\right)
\end{eqnarray*}

\noindent we obtain \eqref{asymptotic.ai_n.z}. By an analog computation we have \eqref{asymptotic.ai_n.jz}, \eqref{asymptotic.ai'_n.z} and \eqref{asymptotic.ai'_n.jz}.

From those asymptotic estimates, we deduce immediately uniform bounds for ${\rm Ai_n} $ and the time derivative ${\rm Ai_n}'$.

\end{proof}

Thanks to the previous Lemma, we end the proof of Lemma \ref{lemma.growth.propa.airy} by getting the upper bound of the propagator $\widetilde{V}_n(s',s)$. Combining the expression of $\widetilde{V}_n$ in function of ${\rm Ai_n}$ given by \eqref{matrix.U.Ai} with the estimates \eqref{bound.ai_n} and \eqref{bound.ai_n'}, we obtain the upper bounds for the coefficients of the matrix flow $\widetilde{V}_n(s',s)$, for $0 \leq s' \leq s \leq \e^{-2/3}$:
\be
	\left\{
		\begin{array}{ccl}
			\left|\widetilde{V}_{n,1,1}(s',s)\right| & \approx & s'^{1/4}(1+s)^{-1/4}\exp\left(|n|\g_0 \frac{2}{3}(s^{3/2}-s'^{3/2})\right) \\[3mm]
			\left|\widetilde{V}_{n,2,1}(s',s)\right| & \approx & \e^{1/3}s'^{1/4}s^{1/4}\exp\left(|n|\g_0 \frac{2}{3}(s^{3/2}-s'^{3/2})\right) \\[3mm]
			\left|\widetilde{V}_{n,1,2}(s',s)\right| & \approx & \e^{-1/3}(1+s')^{-1/4}(1+s)^{-1/4}\exp\left(|n|\g_0 \frac{2}{3}(s^{3/2}-s'^{3/2})\right) \\[3mm]
			\left|\widetilde{V}_{n,2,2}(s',s)\right| & \approx & (1+s')^{-1/4}s^{1/4}\exp\left(|n|\g_0 \frac{2}{3}(s^{3/2}-s'^{3/2})\right)
		\end{array}
	\right.
\ee

\noindent As 
$$
	\e^{-1/3}(1+s)^{-1/4}> s^{1/4} \qquad \forall 0 \leq s < \e^{-2/3}
$$

\noindent we obtain the upper bound for the propagator
$$
	|\widetilde{V}_n(s',s)| \lesssim \e^{-1/3}(1+s')^{-1/4}(1+s)^{-1/4}\exp(|n|\g_0 (2/3)(s^{3/2}-s'^{3/2})) \quad \forall 0\leq s' \leq s < \e^{-2/3}
$$

\noindent which implies \eqref{growth.propa.airy} and ends the proof of Lemma \ref{lemma.growth.propa.airy}.

%%%%%%%%%%%%%%%%%%%%%%%%%%%%%%%%%%%%%%%%%%%%%%%%%%%%%%%%%%%%%%%%%%%%%%%%%%%%%%%

\subsection{Growth of the free solution: proof of Lemma (3.6)}
\label{subsection.free}

%%%%%%%%%%%%%%%%%%%%%%%%%%%%%%%%%%%%%%%%%%%%%%%%%%%%%%%%%%%%%%%%%%%%%%%%%%%%%%%

We prove here Lemma \ref{lemma.growth.freesolution.airy}, following the proof of Lemma \ref{lemma.growth.freesolution.smooth}. We showed in it that it suffices to prove the lower bound for $\widetilde{V}_n$. Thanks to the equalities \eqref{matrix.U.Ai}, a simple computation gives us 
$$
	\widetilde{V}_{+1}(0,s)
	\begin{pmatrix}
		{\rm Ai_1}(0) \\
		-i\e^{1/3} j{\rm Ai_1}'(0)
	\end{pmatrix} = 
	\begin{pmatrix}
		{\rm Ai_1}(js) \\
		-i\e^{1/3} j{\rm Ai_1}'(js)
	\end{pmatrix}
$$

\noindent and also
$$
	\widetilde{V}_{-1}(0,s)
	\begin{pmatrix}
		{\rm Ai_1}(0) \\
		i\e^{1/3} j{\rm Ai_1}'(0)
	\end{pmatrix} = 
	\begin{pmatrix}
		{\rm Ai_1}(js) \\
		i\e^{1/3} j{\rm Ai_1}'(js)
	\end{pmatrix} .
$$

\noindent We denote 
\be
	\label{local.widetilde.f}
	\widetilde{\f}(s,\theta) = {\rm Re}\left(\widetilde{V}_{+1}(0,s)
		\begin{pmatrix}
			{\rm Ai_1}(0) \\
			- i\e^{1/3} j{\rm Ai_1}'(0)
		\end{pmatrix}
		e^{i\theta} + \widetilde{V}_{-1}(0,s)
		\begin{pmatrix}
			{\rm Ai_1}(0) \\
			i\e^{1/3} j{\rm Ai_1}'(0)
		\end{pmatrix}
		e^{-i\theta} \right).
\ee

\noindent and we compute
\be
	\label{free.solution.airy.exact}
	\widetilde{\f}(s,\theta) = 2{\rm Re}
	\begin{pmatrix}
		{\rm Ai_1}(js) \cos(\theta) \\
		-\e^{1/3} j{\rm Ai_1}'(js) \sin(\theta)
	\end{pmatrix} .
\ee

\noindent Next we denote $\widetilde{\f}_1(s)$ and $\widetilde{\f}_2(s)$ the two components of the vector $\widetilde{\f}$ defined by \eqref{local.widetilde.f}. Thanks to Lemma \ref{lemma.asymptotics}, we have
$$
	|\widetilde{\f}_1(s) | \approx {\rm Ai_1}(js) \approx s^{-1/4} \exp(\g_0 2/3s^{3/2}) 
$$

\noindent and

$$
	|\widetilde{\f}_2(s) | \approx \e^{1/3} {\rm Ai_1}'(js) \approx \e^{1/3}s^{1/4} \exp(\g_0 2/3s^{3/2})
$$

\noindent for $1\leq s <\e^{-2/3}$. As $ \e^{1/3}s^{1/4} < s^{-1/4} $ for $s<\e^{-2/3}$, the $L^{\infty}$ norm of the vector $\widetilde{\f}(s,\theta)$ satisfies 
$$
	|\widetilde{\f}(s,\theta)|_{\infty} = |\widetilde{\f}_1(s,\theta)| \approx s^{-1/4} \exp(\g_0 2/3s^{3/2}) 
$$

\noindent for all $1 \leq s <\e^{-2/3}$. Using the same steps as in the proof of Lemma \ref{lemma.growth.freesolution.smooth}, this suffices to end the proof of Lemma \ref{lemma.growth.freesolution.airy}.

%%%%%%%%%%%%%%%%%%%%%%%%%%%%%%%%%%%%%%%%%%%%%%%%%%%%%%%%%%%%%%%%%%%%%%%%%%%%%%%%%%%%%%%%%%%%%%%%%%%%%%%%%%%%%%%%%%%%%%%%
%%%%%%%%%%%%%%%%%%%%%%%%%%%%%%%%%%%%%%%%%%%%%%%%%%%%%%%%%%%%%%%%%%%%%%%%%%%%%%%%%%%%%%%%%%%%%%%%%%%%%%%%%%%%%%%%%%%%%%%%

 %Pour la bibliographie
	%\newpage
	%\nocite{*}
	\bibliographystyle{alpha}
	\bibliography{onset_insta}

%%%%%%%%%%%%%%%%%%%%%%%%%%%%%%%%%%%%%%%%%%%%%%%%%%%%%%%%%%%%%%%%%%%%%%%%%%%%%%%%%%%%%%%%%%%%%%%%%%%%%%%%%%%%%%%%%%%%%%%%%%%%%%%%%%%%%%%%%%%%%%%%%%%%%%%%%%%%%%%%%%%%%%%%%%%%%%%%%%%%%%%%%%%%%%%%%%%%%%%%%%%%%%%%%%%%%%%%%%%%%%%%%%%%%%%%%%%%%%%%

\end{document}